\documentclass{siamltex}
\usepackage{amsfonts,amsmath,amssymb,color,verbatim}
\usepackage{stmaryrd}
\usepackage{hyperref}
\usepackage[mathscr]{eucal}
\usepackage{accents}
\usepackage{pgfplots}
\usepackage{ifoddpage}
\usepackage{marginnote}
\usepackage{float} 
\usepackage{placeins} 
\usetikzlibrary{patterns}
\usepackage{wasysym}
\usepackage[ruled,vlined]{algorithm2e}
\usepackage[colorinlistoftodos,textwidth=4cm,shadow]{todonotes}
\usepackage{bm}
\usepackage{tikz}
\usetikzlibrary{arrows,snakes,backgrounds}
\usepackage{enumitem}
\usepackage{amsmath}
\usepackage{pgf}
\usepackage{multirow}
\usepackage{caption}
\usepackage{subcaption}
\usepackage{url}
\usepackage{mathtools}
\usepackage{soul}
\usepackage{xcolor}
\usepackage{mathtools}
\usepackage{amsmath}

\newtheorem{hypothesis}{Hypothesis}
\makeatletter
\def\BState{\State\hskip-\ALG@thistlm}
\makeatother

\definecolor{royalblue}{rgb}{0.0, 0.0, 0.0}
\newcommand{\bsp}[3]{\textbf{bsp}(#1)^{#2\times #2}_{#3}}

\newtheorem{remark}{Remark}[section]

\title{``Compress and eliminate'' solver for symmetric positive definite sparse matrices\footnotemark[4]}
\author{Daria A.~Sushnikova\footnotemark[5] \and Ivan V.~Oseledets\footnotemark[3]\ \footnotemark[5]}
\begin{document}

\maketitle
\renewcommand{\thefootnote}{\fnsymbol{footnote}}

\footnotetext[3]{Skolkovo Institute of Science and Technology,
Nobel St.~3, Skolkovo Innovation Center, Moscow, 143025
Moscow Region, Russia (i.oseledets@skolkovotech.ru)}
\footnotetext[5]{Institute of Numerical Mathematics Russian Academy of Sciences,
Gubkina St. 8, 119333 Moscow, Russia}
\footnotetext[4]{The work was supported by Russian Foundation of Basic Research grant 17-01-00854.}

\renewcommand{\thefootnote}{\arabic{footnote}}

\begin{abstract}
We propose a new approximate factorization for solving linear systems with symmetric positive definite sparse matrices.
In a nutshell the algorithm is to apply hierarchically block Gaussian elimination and additionally compress the fill-in.
The systems that have efficient compression of the fill-in mostly arise from discretization of partial differential equations.
We show that the resulting factorization can be used as an efficient preconditioner and compare the proposed approach with the state-of-art direct and iterative solvers.

\end{abstract}



\pagestyle{myheadings} \thispagestyle{plain}

\section{Introduction}
We consider a linear system with a large sparse symmetric positive definite (SPD) matrix $A = A^{\top} > 0$, $A \in \mathbb{R}^{N \times N}$,
$$Ax = b. $$
Such systems typically come from the discretization of partial differential equations. It is well known that sparse Cholesky factorization has $\mathcal{O}(N)$ complexity only for very simple problems (for example, discretization of one-dimensional PDEs). As an alternative, hierarchical low-rank approximations have been used to approximate big dense blocks in the Cholesky factors. These methods can achieve optimal complexity for large classes of sparse linear systems. However the constant in $\mathcal{O}(N)$ can be quite large, thus for practically interesting values of $N$ they often lose to ``purely sparse'' software  (CHOLMOD, PARDISO, UMFPACK).

In this paper we try to combine those approaches and propose a new algorithm for (approximate) Cholesky-type factorization of sparse matrices. We start the standard block Gaussian elimination, but after each step we apply row and column transformation to the new fill-ins in the so-called ``far zone'' to create new zeros.  In this scheme we have discovered experimentally that it is possible to eliminate approximately half of the unknowns without introducing new non-zeros. The approach is then applied in the multilevel fashion. The most time-consuming part of the method is low-rank approximation of
small matrices, which has to be done after each elimination step.
Then we show that the resulting factorization can be used as an efficient preconditioner, and compare the efficiency with the state-of-the-art fast direct sparse solver CHOLMOD on a model example.
\section{Algorithm}
Before running the algorithm we select the permutation $P$ and work with the matrix {\color{royalblue}
\begin{equation}
A_0 = PAP^{\top}.
\label{eq:init_perm}
\end{equation}} This permutation is crucial for the algorithm. The choice of the permutation, its properties and its influence on the proposed algorithm is discussed in Section \ref{sec:complan}. The permutation $P$ and the block size $B$ define the division of the matrix $A$ into sub-blocks. The block size $B$ is the parameter of the algorithm and is typically small, the selection of the block size is discussed later. \\

\begin{figure}[H]
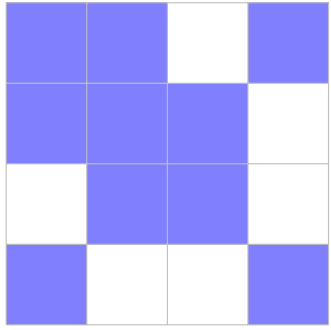

\centering
\resizebox{.34\textwidth}{!}{
  \tikz{
	\foreach \i in {0,...,3}{
		\fill[blue!50!white] (\i,3-\i) rectangle (\i+1,4-\i);
	}
	\fill[blue!50!white] (0,2) rectangle (1,3);
	\fill[blue!50!white] (0,0) rectangle (1,1);
	\fill[blue!50!white] (1,1) rectangle (2,2);
	\fill[blue!50!white] (1,3) rectangle (2,4);
	\fill[blue!50!white] (2,2) rectangle (3,3);
	\fill[blue!50!white] (3,3) rectangle (4,4);
	\draw[step=1cm,gray!50!white] (0,0) grid (4,4);

}}
\caption{Matrix $A_0$}
\label{fig:a0}
\end{figure}
We represent the matrix {\color{royalblue}$A_0$} in the \emph{compressed sparse block} (CSB) \cite{saad-csb-1994} format.
The nonzero blocks are referred to as ``close'' blocks. The nonzero blocks that arise during the elimination are called ``far'' blocks{\footnote{{``Far'' blocks are typically called ``fill'' blocks, but we use ``close'' and ``far'' notation to emphasize the connection with the FMM method.}}}. The first step is to eliminate the first block row.

\subsection{Elimination of the first block row}
Consider the matrix {\color{royalblue}$A_0$} in the following block form:

$${\color{royalblue}A_0} =
\begin{bmatrix}
D_1 &A_{1r}  \\
A_{1r}^{\top}  & A_{1*}
\end{bmatrix},
 $$
where $D_1$ is $B \times B$. The first block row $\widetilde{R}_1$ can be written as
$$
   \widetilde{R}_1 =  \begin{bmatrix}D_1 & A_{1r} \end{bmatrix} = \begin{bmatrix}D_1 & C_1 & 0 \end{bmatrix} P_1^{\text{col}},
$$
where $C_1$ corresponds to all close blocks. The matrix $P_1^{\text{col}}$ is a permutation matrix.
To get the first step of the block Cholesky factorization we factorize the diagonal block:
$$
D_1 = \widehat{L}_1 \widehat{L}_1^{\top}.
$$
Then we apply the block Cholesky factorization \cite{Davis-cholmod-2008}:

\begin{equation}
\label{eq:bl_chol}
{\color{royalblue}A_0} = \begin{bmatrix}
\widehat{L}_1 & 0 \\
 A_{1r}^{\top}\widehat{L}_1^{-1}  & 0
\end{bmatrix}  \begin{bmatrix}
\widehat{L}_1^{\top} & \widehat{L}_1^{-\top}A_{1r} \\
0  & 0
\end{bmatrix} + \begin{bmatrix}
0 & 0 \\
0  & \widetilde{A}_1^*
\end{bmatrix}=  \widetilde{L}_1\widetilde{L}_1^{\top} + \widetilde{A}_1,
\end{equation}
where the block  $\widetilde{A}_1^* = A_{1*} - A_{1r}^{\top}D_1^{-1}A_{1r}$ is the Schur complement.
The matrix $\widetilde{A}_1^*$ has the sparsity pattern equal to the union of sparsity patterns of the matrix $A_{1*}$  and the fill-in matrix
\begin{equation}
\widetilde{A}_{1F} = A_{1r}^{\top} D_1^{-1}A_{1r}.
\label{eq:A_st}
\end{equation}
New non-zero blocks can appear in positions $(i,j)$, where $i \in \mathcal{I}_1, j \in \mathcal{I}_1${, and $ \mathcal{I}_1$ is the set of indices} of nonzero blocks in the first block row $\widetilde{R}_1,$ see Figure  \ref{fig:fill0}.
\begin{figure}[H]
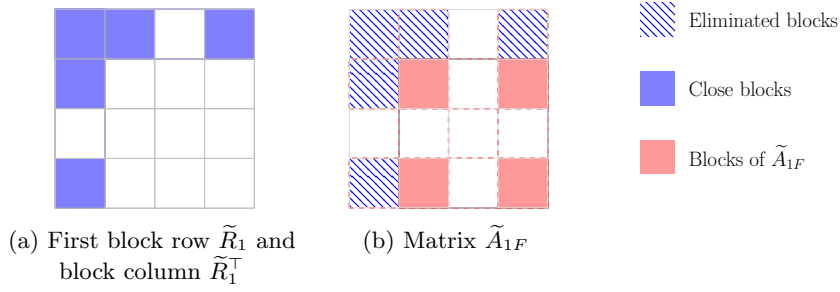

\begin{subfigure}[t]{.3\textwidth}
\centering
\resizebox{.7\textwidth}{!}{
  \tikz{
	\fill[blue!50!white] (0,3) rectangle (1,4);
	\fill[blue!50!white] (0,2) rectangle (1,3);
	\fill[blue!50!white] (0,0) rectangle (1,1);
	\fill[blue!50!white] (1,3) rectangle (2,4);
	\fill[blue!50!white] (3,3) rectangle (4,4);
    \draw[blue] (0,0) rectangle (1,4);
	\draw[blue] (0,3) rectangle (4,4);
	\draw[step=1cm,gray!50!white] (0,0) grid (4,4);
	\draw[step=1cm,gray!50!white] (0,0) grid (4,4);
}}
    \caption{ \centering First block row $\widetilde{R}_1$ and block column $\widetilde{R}_1^{\top}$}
\end{subfigure}%
\begin{subfigure}[t]{.3\textwidth}
\centering
\resizebox{.7\textwidth}{!}{
\tikz{
	\draw[white, pattern=north west lines, pattern color=blue] (0,3) rectangle (1,4);
	\draw[white, pattern=north west lines, pattern color=blue] (0,2) rectangle (1,3);
	\draw[white, pattern=north west lines, pattern color=blue] (0,0) rectangle (1,1);
	\draw[white, pattern=north west lines, pattern color=blue] (1,3) rectangle (2,4);
	\draw[white, pattern=north west lines, pattern color=blue] (3,3) rectangle (4,4);
	\draw[black] (1,0) rectangle (4,3);
    \draw[red,thick,dashed] (0,0) rectangle (4,1);
	\draw[red,thick,dashed] (0,2)rectangle (4,3);
	\draw[red,thick,dashed] (1,0) rectangle (2,4);
	\draw[red,thick,dashed] (3,0) rectangle (4,4);
	\fill[red!40!white] (1,0) rectangle (2,1);
	\fill[red!40!white] (1,2) rectangle (2,3);
	\fill[red!40!white] (3,0) rectangle (4,1);
	\fill[red!40!white] (3,2) rectangle (4,3);

	\draw[step=1cm,gray!30!white,very thin] (0,0) grid (4,4);
}}
\caption{Matrix $\widetilde{A}_{1F}$}
\label{fig:fill0}
\end{subfigure}%
\begin{subfigure}[t]{.3\textwidth}
\centering
\resizebox{.7\textwidth}{!}{
  \tikz{
	\fill[blue!50!white] (0,4) rectangle (1,5);
	\node [right] at (1.28,4.5) {\LARGE Close blocks};
    \fill[red!40!white]  (0,2) rectangle (1,3);
	\node [right] at (1.28,2.5) {\LARGE {Blocks of $\widetilde{A}_{1F}$}};
    \fill[pattern=north west lines, pattern color=blue] (0,6) rectangle (1,7);
	\node [right] at (1.28,6.5) {\LARGE Eliminated blocks};
    \fill[white] (0,1) rectangle (1,2);
}}
\end{subfigure}%
\centering
\caption{Illustration of the fill-in caused by the first block row elimination}
\label{fig:fill1}
\end{figure}$\\$

Note that the number of non-zero blocks in the new matrix $\widetilde{A}_{1F}$ is bounded by $(\#\mathcal{I}_1)^2$.
The positions of those blocks are known in advance.
If the elimination is continued, the number of far blocks may grow. To avoid this {problem we use an additional \emph{compression procedure.}}

\subsection{Compression and elimination of $i$-th block row}
\label{sec:ce}

Assume that we have already eliminated $(i-1)$ block rows and now we are working with the $i$-th block row, see Figure~\ref{fig:elim0}.
\begin{figure}[H]
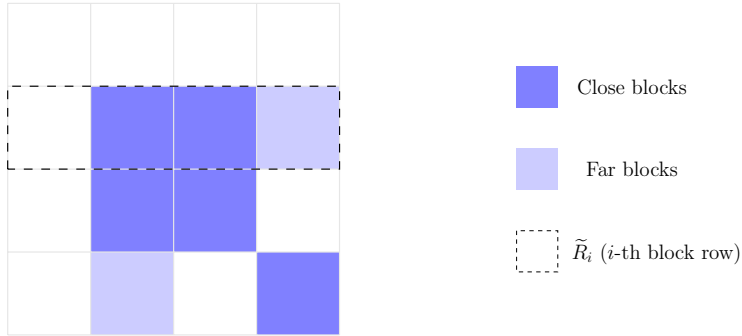

\centering
\begin{subfigure}{.7\textwidth}
\centering
\resizebox{.5\textwidth}{!}{
\tikz{
	\foreach \i in {1,...,3}{
		\fill[blue!50!white] (\i,3-\i) rectangle (\i+1,4-\i);
	}
	\fill[blue!50!white] (1,1) rectangle (2,2);
	\fill[blue!50!white] (2,2) rectangle (3,3);
	\fill[blue!20!white] (1,0) rectangle (2,1) ;
	\fill[blue!20!white] (3,2) rectangle (4,3) ;
	\draw[step=1cm,gray!20!white] (0,0) grid (4,4);
	\draw[black,dashed] (0,2) rectangle (4,3) ;
}}
\end{subfigure}%
\begin{subfigure}{.3\textwidth}
\resizebox{.8\textwidth}{!}{
  \tikz{
	\fill[blue!50!white] (0,0) rectangle (1,1);
	\node at (2.8,0.5) {\Large Close blocks };
    \fill[blue!20!white] (0,-1) rectangle (1,-2);
	\node at (2.8,-1.5) {\Large Far blocks };
    \node at (3.4,-3.5) {\Large $\widetilde{R}_i$ ({$i$-th block row)}};
    \draw[black,dashed] (0,-3) rectangle (1,-4) ;
}}
  \label{fig:1}
\end{subfigure}%
\caption{ Matrix $\widetilde{A}_{i-1}$, starting point for general elimination ($i = 3$)}
\label{fig:elim0}
\end{figure}$\\$
Here we consider the $i$-th block row of the matrix $\widetilde{A}_{i-1}$ and try to eliminate it. After all previous steps we obtain some
{ far blocks},
see Figure~\ref{fig:elim0}.
We can reorder columns in the block row,
$$\widetilde{R}_i = \begin{bmatrix} D_i&C_i&F_i&0\end{bmatrix}P^{\text{col}}_i,$$
where $P^{\text{col}}_i$ is the permutation matrix,
$C_i = [C_i^{(1)} \dots ~ C_i^{(N^{C}_i)}]$ are close blocks in the block row $\widetilde{R}_i$, $N^{C}_i$ is the number of close blocks in  $\widetilde{R}_i$,
$F_i = [F_i^{(1)} \dots ~ F_i^{(N^{F}_i)}]$ are far blocks in the block row $\widetilde{R}_i$, $N^{F}_i $ is the number of far blocks in  $\widetilde{R}_i$,
see Figure~\ref{fig:row3}.

The key idea is to approximate the matrix $F_i \in \mathbb{R}^{B \times (BN^{F}_i)}$ by a low-rank matrix:
\begin{equation}\widetilde{U}_i^{\top}F_i = \begin{bmatrix}\hat{F}_i\\ E_i\end{bmatrix},\end{equation}\label{eq:rreav} where $\widetilde{U}_i$ is a unitary matrix, $\|E_i\| < \varepsilon $ for some $ \varepsilon$ and
$\hat{F}_i \in \mathbb{R}^{r \times BN^{F}_i}$.

{
\begin{remark}
Low-rank approximation presented in equation~\eqref{eq:rreav} is a crucial part of the proposed algorithm.
We consider two strategies of finding the rank~$r$:
\begin{itemize}
\item CE($\varepsilon$): for a given accuracy $\varepsilon$ find the minimum rank $r$ such that  $\|E_i\| < \varepsilon$ (adaptive low-rank approximation).

\item CE($r$): Fix the rank~$r$, typically $r = \frac{B}{2}$ (fixed rank approximation).
\end{itemize}
CE($\varepsilon$) algorithm leads to a better factorization accuracy and thus can be used as an approximate direct solver. A major drawback of the CE($\varepsilon$) algorithm is that the compression of the far blocks is not guaranteed, which may lead to a memory-intensive factorization.

While CE($r$) algorithm directly controls the memory usage, it lacks control over the factorization accuracy. Therefore, CE($r$) is unsuitable for direct approximate solution of the system, but can be a good preconditioner.
\end{remark}}
$\\$
Let \begin{equation} \widetilde{Q}_i =
\begin{bmatrix}
I_{(i-1)B} &0 & 0 \\
0&\widetilde{U}_i &0  \\
0&0 & I_{(M-i)B}
\end{bmatrix},
\label{eq:q}
\end{equation}
then the matrix
$$ {\color{royalblue}\widehat{A}}_{i-1} \approx \widetilde{Q}^{\top}_i \widetilde{A}_{i-1} \widetilde{Q}_i$$ has the $i$-th row
$$ \widetilde{R}_{i*} = \begin{bmatrix}\widetilde{U}_i^{\top}D_i\widetilde{U}_i&\widetilde{U}_i^{\top}C_i& \begin{bmatrix}\hat{F}_i\\0\end{bmatrix}&0\end{bmatrix}P^{\text{col}}_i,$$ see illustration in Figure~\ref{fig:row3}.

The crucial part of the algorithm is that we eliminate only a part of the $i$-th block row of the modified matrix ${\color{royalblue}\widehat{A}}_{i-1}$, which has zero elements in the far zone. Introduce the block subrow $\widetilde{S}_i \in \mathbb{R}^{(B - r)\times N}$ (red-dashed in Figure~\ref{fig:row3}):
$$\widetilde{S}_i = \begin{bmatrix}0_{(B-r) \times r} & I_{(B-r)}\end{bmatrix}\widetilde{R}_{i*}.$$
Now we can eliminate the block subrow $ $ from the matrix ${\color{royalblue}\widehat{A}}_{i-1}$ using the block Cholesky elimination.
\begin{figure}[H]
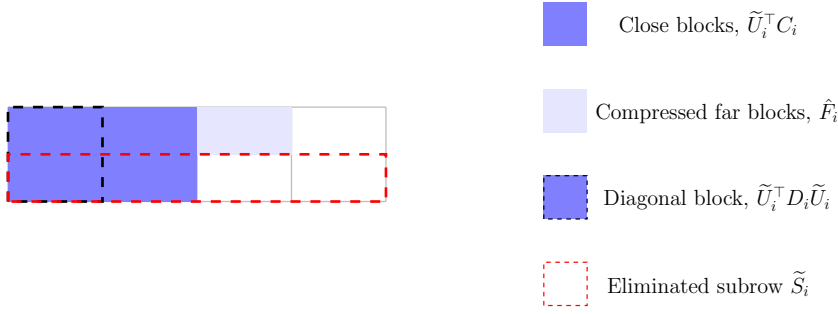

\centering
\begin{subfigure}{.5\textwidth}
  \centering
  \resizebox{.8\textwidth}{!}{
\tikz{
	\draw[step=1cm,gray!50!white] (0,0) grid (4,1);
	\fill[blue!10!white] (2,0.5) rectangle (3,1);
	\fill[blue!50!white] (0,0) rectangle (2,1);
    \draw[black,thick,dashed] (0,0) rectangle (1,1);
    \draw[red,thick,dashed] (0,0) rectangle (4,0.5);
	}}
  \label{fig:sub1}
\end{subfigure}%
\begin{subfigure}{.5\textwidth}
  \centering
    \hspace{-1cm}
  \resizebox{.8\textwidth}{!}{
  \tikz{
	\draw[step=1cm,white,very thin] (0,0) grid (-2,5);
	\fill[blue!10!white] (0,2) rectangle (1,3);
	\fill[blue!50!white] (0,4) rectangle (1,5);
    \draw[red,thick,dashed] (0,-2) rectangle (1,-1);
    \node at (4,0.5) {\Large Diagonal block, $\widetilde{U}_i^{\top}D_i\widetilde{U}_i$};
    \node at (4,2.5) {\Large Compressed far blocks, $\hat{F}_i$};
    \node at (3.8,4.5) {\Large Close blocks, $\widetilde{U}_i^{\top}C_i$};
    \node at (3.8,-1.5) {\Large Eliminated subrow $\widetilde{S}_i$};
    \fill[blue!50!white] (0,0) rectangle (1,1);
    \draw[black,thick,dashed](0,0) rectangle (1,1);
	}}
  \label{fig:1}
\end{subfigure}%
\caption{Compression step}
\label{fig:row3}
\end{figure}$\\$
Denote the diagonal block of the block subrow $\widetilde{S}_i$ as $\widetilde{S}_{ii}$. If the Cholesky decomposition of the diagonal block is $\widetilde{S}_{ii} = \widehat{L}_i\widehat{L}_i^{\top}$, then
$$ {\color{royalblue}\widehat{A}}_{i-1} = \widetilde{L}_i\widetilde{L}_i^{\top} + \widetilde{A}_i, $$ where the matrix $\widetilde{A}_i \in \mathbb{R}^{N\times N}$ has zeros instead of subrow $\widetilde{S}_i$ and subcolumn $\widetilde{S}_i^{\top}$, and
\begin{equation}
\widetilde{L}_i = \widetilde{S}_i^{\top}\widehat{L}_i^{-\top}.
\label{eq:sm_L}
\end{equation}
Note that in this scheme the block sparsity pattern of the block subrow~$\widetilde{S}_i$ coincides with the block sparsity pattern of the original matrix $A$, thus after one elimination in the matrix~$\widetilde{A}_i$ only blocks from the block sparsity pattern of the matrix {\color{royalblue}$\widetilde{A}^2_i$}  can arise. Also note that the compression step does not affect locations and sizes of the far blocks, see Figure~\ref{fig:fill2}. {There is an important interpretation of the far blocks: far blocks sparsity pattern consists of blocks that are nonzero in the matrix $A^2$ and zero in the matrix $A$.}
\begin{figure}[H]
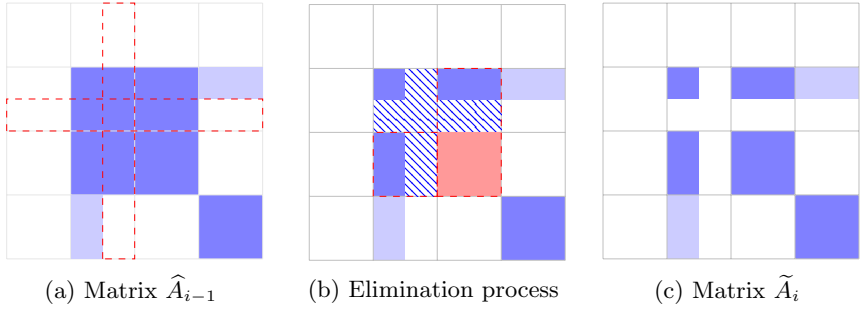

\begin{subfigure}{.3\textwidth}
\centering
  \resizebox{.9\textwidth}{!}{
\tikz{
	\foreach \i in {1,...,3}{
		\fill[blue!50!white] (\i,3-\i) rectangle (\i+1,4-\i);
	}
	\fill[blue!50!white] (1,1) rectangle (2,2);
	\fill[blue!50!white] (2,2) rectangle (3,3);
	\draw[step=1cm,gray!20!white] (0,0) grid (4,4);
	\fill[blue!20!white] (1,0) rectangle (1.5,1) ;
    \fill[blue!20!white] (3,2.5) rectangle (4,3) ;
    \draw[red,dashed] (0,2) rectangle (4,2.5);
    \draw[red,dashed] (1.5,0) rectangle (2,4);
}}
\caption{Matrix ${\color{royalblue}\widehat{A}}_{i-1}$}
\label{fig:elim2}
\end{subfigure}
\begin{subfigure}{.3\textwidth}
\centering
\resizebox{.9\textwidth}{!}{
\tikz{

	\foreach \i in {1,...,3}{
		\fill[blue!50!white] (\i,3-\i) rectangle (\i+1,4-\i);
	}
	\fill[blue!50!white] (1,1) rectangle (2,2);
	\fill[blue!50!white] (2,2) rectangle (3,3);
	\fill[blue!20!white] (1,0) rectangle (1.5,1) ;
	\fill[blue!20!white] (3,2.5) rectangle (4,3) ;
    \fill[white]  (0,2) rectangle (4,2.5);
    \fill[white]  (1.5,0) rectangle (2,4);
	\fill[pattern=north west lines, pattern color=blue] (1,2) rectangle (2,3);
	\fill[blue!50!white] (1,2.5) rectangle (1.5,3) ;
    \draw[step=1cm,gray!50!white] (0,0) grid (4,4);
    \fill[pattern=north west lines, pattern color=blue] (2,2) rectangle (3,2.5);
    \fill[pattern=north west lines, pattern color=blue] (1.5,1) rectangle (2,2);
    \fill[red!40!white]  (2,1) rectangle (3,2);
	\draw[red,dashed] (1,1) rectangle (2,2);
    \draw[red,dashed] (2,1) rectangle (3,3);
}}
\caption{{Elimination process}}
\label{fig:fill2}
\end{subfigure}%
\begin{subfigure}{.3\textwidth}
\centering
  \resizebox{.9\textwidth}{!}{
\tikz{
	\foreach \i in {1,...,3}{
		\fill[blue!50!white] (\i,3-\i) rectangle (\i+1,4-\i);
	}
	\fill[blue!50!white] (1,1) rectangle (2,2);
	\fill[blue!50!white] (2,2) rectangle (3,3);
	\fill[blue!20!white] (1,0) rectangle (1.5,1) ;
	\fill[blue!20!white] (3,2.5) rectangle (4,3) ;
	\fill[white]  (0,2) rectangle (4,2.5);
	\fill[white]  (1.5,0) rectangle (2,4);
	\fill[blue!50!white] (1,2.5) rectangle (1.5,3) ;
	\draw[step=1cm,gray!50!white] (0,0) grid (4,4);
}}
\caption{Matrix $\widetilde{A}_i$}
\label{fig:elim2}
\end{subfigure}%
\caption{Elimination step. Gray color denotes nonzero rows and columns of matrix $\widetilde{A}_{i}.$}
\end{figure}
$\\$
After the compression-elimination step, the matrix $\widetilde{A}_{i-1}$ is approximated as
$$
  \widetilde{Q}_i\widetilde{A}_{i-1}\widetilde{Q}_i^{\top} \approx  \widetilde{L}_i\widetilde{L}_i^{\top} + \widetilde{A}_{i},$$
where $\widetilde{L}_{i} \in \mathbb{R}^{N \times (B-r)}$ and the matrix $\widetilde{A}_{i} \in \mathbb{R}^{N \times N}$ have the same size {as} the initial matrix $A$.

\subsection{Full sweep of the algorithm}
\label{sec:full_one}

{Consider the compression-elimination procedure} after all block rows have been processed. We store factors~$\widetilde{L}_i$ multiplied by corresponding matrices~$\widetilde{Q}_i$ as columns of the matrix~${\color{royalblue}\check{L}}_1$:
\begin{equation}{\color{royalblue}\check{L}}_1 = \begin{bmatrix}\left( \prod_{j = M}^{2} \widetilde{Q}_j\right)\widetilde{L}_1& \cdots&\left(\prod_{j = M}^{i+1}\widetilde{Q}_j\right)\widetilde{L}_i& \cdots & \widetilde{L}_M \end{bmatrix}.
\label{eq:big_L}
\end{equation}
We obtain the following approximation:
$$\widetilde{Q}_M \dots  \widetilde{Q}_1 A \widetilde{Q}_1^{\top}\dots \widetilde{Q}_M^{\top} \approx {\color{royalblue}\check{L}}_1{\color{royalblue}\check{L}}^{\top}_1 + {\color{royalblue}\check{A}}_1,$$
where the matrix ${\color{royalblue}\check{L}}_1 \in \mathbb{R}^{N\times N_{L1}},$ see Figure~\ref{fig:L_1}, the matrix ${\color{royalblue}\check{A}}_1  \in \mathbb{R}^{N\times N}$, see Figure~\ref{fig:AM}.
Since multiplication by the matrix $\widetilde{Q}_i$ does not change the sparsity patterns of the matrices~${\color{royalblue}\check{L}}_1$ and~${\color{royalblue}\check{A}}_1$ during the elimination, the block sparsity pattern of~${\color{royalblue}\check{L}}_1$ can be easily obtained from the block sparsity pattern of the original matrix.
Denote
$$Q_1 = \prod_{i=M}^{1}\widetilde{Q}_i.$$
Due to a special structure of $\widetilde{Q}_i$ \eqref{eq:q},
$$Q_1 = \text{diag}(\widetilde{U}_1, \dots ,\widetilde{U}_M),$$ i.e. $Q_1$ is a block-diagonal matrix.
\begin{figure}[H]
\centering
\begin{subfigure}[t]{.3\textwidth}
\centering
\resizebox{.8\textwidth}{!}{
\tikz{
	\draw[step=1cm,white,very thin] (-0.5,0) grid (3.5,4);
	\draw[step=0.5cm,gray!50!white] (0,0) grid (2.5,4);
	\fill[green!10!white] (0,0) rectangle (1,4);
	\foreach \i in {1,...,3}{
		\fill[green!10!white] (\i*0.5+0.5,0) rectangle (\i*0.5+1,4-\i);
	}
	\draw[green!50!black,fill=green!50!white] (0,3) rectangle (1,4);
	\foreach \i in {1,...,3}{
		\draw[green!50!black,fill=green!50!white] (\i*0.5+0.5,3-\i) rectangle (\i*0.5+1,4-\i);
	}
	\draw[green!50!black,fill=green!50!white] (1.5,2.5) rectangle (2,3);
	\draw[green!50!black,fill=green!50!white] (0,0) rectangle (1,1);
	\draw[green!50!black,fill=green!50!white] (0,2) rectangle (1,3);
	\draw[green!50!black,fill=green!50!white] (1,1) rectangle (1.5,2);
}}
\caption{Matrix ${\color{royalblue}\check{L}}_1$}
\label{fig:L_1}
\end{subfigure}%
\begin{subfigure}[t]{.3\textwidth}
\centering
    \hspace{-0.5cm}
  \resizebox{.8\textwidth}{!}{
  \tikz{
	\draw[step=1cm,gray!50!white] (0,0) grid (4,4);
	\foreach \i in {0,...,3}{
		\fill[blue!40!red!30!white] (\i,3-\i) rectangle (\i+1,4-\i);
	}
}}
\caption{Matrix $Q_1$}
\end{subfigure}%
\begin{subfigure}[t]{.3\textwidth}
\centering
    \hspace{-1cm}
  \resizebox{.8\textwidth}{!}{
  \tikz{
    \draw[step=1cm,gray!50!white] (0,0) grid (4,4);
	\foreach \i in {1,...,3}{
		\fill[blue!50!white] (\i,3.5-\i) rectangle (\i+0.5,4-\i);
	}
	\fill[blue!50!white] (1,1.5) rectangle (1.5,2);
	\fill[blue!50!white] (2,2.5) rectangle (2.5,3);
	\fill[blue!20!white] (1,0.5) rectangle (1.5,1);
	\fill[blue!20!white] (3,2.5) rectangle (3.5,3);
}}
\caption{Matrix ${\color{royalblue}\check{A}}_1$}
\label{fig:AM}
\end{subfigure}
\caption{Illustration of matrices ${\color{royalblue}\check{L}}_1$, $Q_1$ and ${\color{royalblue}\check{A}}_1$}
\label{fig:LQ}
\end{figure}$\\$

Let us now permute the block rows of the matrix {\color{royalblue}$A_0$} in such a way that the eliminated rows are before the non-eliminated ones:

\begin{equation}\label{eq:rem1} P_1 Q_1 {\color{royalblue}A_0} Q_1^{\top}P_1^{\top} \approx P_1({\color{royalblue}\check{L}}_1{\color{royalblue}\check{L}}_1^{\top} + {\color{royalblue}\check{A}}_1)P_1^{\top} = L_1 L^{\top}_1 + \begin{bmatrix}
0_{n_{L1} \times n_{L1}} & 0 \\
0&A_1
\end{bmatrix},\end{equation}
where
$$L_1 =  P_1{\color{royalblue}\check{L}}_1.$$
The matrix $L_1\in \mathbb{R}^{N\times N_{L1}}$ is  a sparse block lower triangular matrix with  the block size $(B-r) \times (B-r)$, $A_1  \in \mathbb{R}^{N_{A_1}\times N_{A_1}}$ is a block-sparse matrix $A_1 = P_1{\color{royalblue}\check{A}}_1P_1^{\top}$ with  the block size $r\times r$  and $N = N_{L_1}+N_{A_1},$ see Figure~\ref{fig:AM1}.
\begin{figure}[H]
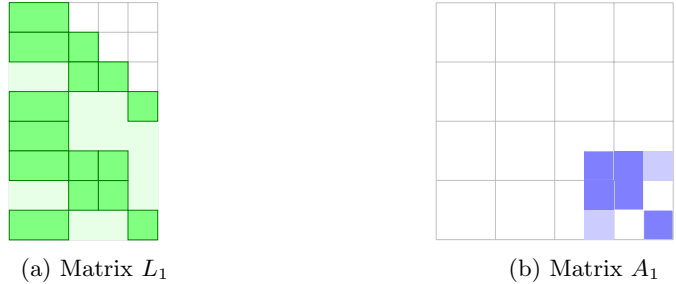

\centering
\begin{subfigure}{.5\textwidth}
  \centering
  \resizebox{.5\textwidth}{!}{
  \tikz{
  \draw[step=1cm,white,very thin] (-0.5,0) grid (3.5,4);
  \draw[step=0.5cm,gray!50!white] (0,0) grid (2.5,4);
  \fill[green!10!white] (0,0) rectangle (1,4);
  \foreach \i in {1,...,3}{
	  \fill[green!10!white] (\i*0.5+0.5,0) rectangle (\i*0.5+1,4-\i*0.5);
  }
  \draw[green!50!black,fill=green!50!white] (0,3.5) rectangle (1,4);
  \draw[green!50!black,fill=green!50!white] (0,1.5) rectangle (1,2);
  \foreach \i in {1,...,3}{
	  \draw[green!50!black,fill=green!50!white] (\i*0.5+0.5,3.5-\i*0.5) rectangle (\i*0.5+1,4-\i*0.5);
	  \draw[green!50!black,fill=green!50!white] (\i*0.5+0.5,1.5-\i*0.5) rectangle (\i*0.5+1,2-\i*0.5);
  }
  \draw[green!50!black,fill=green!50!white] (0,0) rectangle (1,0.5);
  \draw[green!50!black,fill=green!50!white] (0,1) rectangle (1,1.5);
  \draw[green!50!black,fill=green!50!white] (0,2) rectangle (1,2.5);
  \draw[green!50!black,fill=green!50!white] (0,3) rectangle (1,3.5);

  \draw[green!50!black,fill=green!50!white] (1,0.5) rectangle (1.5,1);
  \draw[green!50!black,fill=green!50!white] (1,2.5) rectangle (1.5,3);
  \draw[green!50!black,fill=green!50!white] (1.5,1) rectangle (2,1.5);
}}
  \caption{Matrix $L_1$}
  \label{fig:sub1}
\end{subfigure}%
\begin{subfigure}{.5\textwidth}
  \centering
    \hspace{-1cm}
  \resizebox{.5\textwidth}{!}{
  \tikz{
  \foreach \i in {1,...,3}{
	\fill[blue!50!white] (\i*0.5+2,1.5-\i*0.5) rectangle (\i*0.5+2.5,2-\i*0.5);
	}
	\draw[step=0.5cm,white,very thin] (2.5,0) grid (4,1.5);
    \draw[step=1cm,gray!50!white] (0,0) grid (4,4);
	\fill[blue!50!white] (2.5,0.5) rectangle (3,1);
	\fill[blue!50!white] (3,1) rectangle (3.5,1.5);
	\fill[blue!20!white] (2.5,0) rectangle (3,0.5);
	\fill[blue!20!white] (3.5,1) rectangle (4,1.5);
}}
  \caption{Matrix $A_1$}
  \label{fig:AM1}
\end{subfigure}%
\caption{Matrices $L_1$ and $A_1$ after the full sweep of compression-elimination procedure}
\label{fig:AM1}
\end{figure}$\\$
We refer to the full sweep of the compression-elimination procedure described above as ``one level'' of elimination. Then we consider the matrix $A_1$ as a new block sparse matrix and again apply  the compression-elimination procedure to it (``second level'' of elimination).

{\color{royalblue}
\subsection{Expanded residual graphs}
\label{sec:erg}
In this section we illustrate one level of the compression-elimination procedure using the expanded residual graph toolbox.
The example matrix $A_0$ has an associated graph illustrated in Figure~\ref{fig:erg}.
\begin{figure}[H]
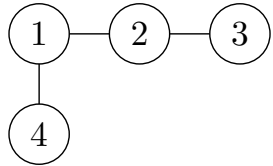

\centering
\resizebox{.28\textwidth}{!}{
\tikz{
\foreach \i in {0,...,2}{
	\draw (\i,0) circle (0.3);
}
\draw (0,-1) circle (0.3);
\node [] at (0,0) {1};
\node [] at (1,0) {2};
\node [] at (2,0) {3};
\node [] at (0,-1) {4};
\draw [] (0.3,0) -- (0.7,0);
\draw [] (1.3,0) -- (1.7,0);
\draw [] (0,-0.7) -- (0,-0.3);
}}
\caption{Connectivity graph of the matrix~$A_0$}
\label{fig:erg}
\end{figure}
In the usual sparse Gaussian elimination framework, we would reason about
sparsity at a graph level by maintaining residual graphs associated with the
Schur complements. At step $k$, the structure of column $k$ is given by connections
in the $k$-th graph; to get the next graph, we connect all neighbors of $k$ in the $k$-th
graph into a clique. The compress-eliminate procedure can be seen as adding
an intermediate operation; we first split the $k$-th node into ``basis'' and ``non-basis''
degrees of freedom, and then only eliminate the ``non-basis'' piece before we move
on to the next node. As a result, at the end of one pass, we have a piece of $L$
with block structure inherited from that of $A$, since the ``non-basis'' nodes only have
the connectivity of the original matrix; and a Schur complement consisting of
the ``basis'' nodes from the pass, which has structure associated with $A^2$. We
illustrate the compress-eliminate process in terms of the graph in Figure~\ref{fig:one_lvl_ce_gr} and
show the resulting structure in Figure~\ref{fig:fig:one_lvl_ce_mat}.

\begin{figure}[H]
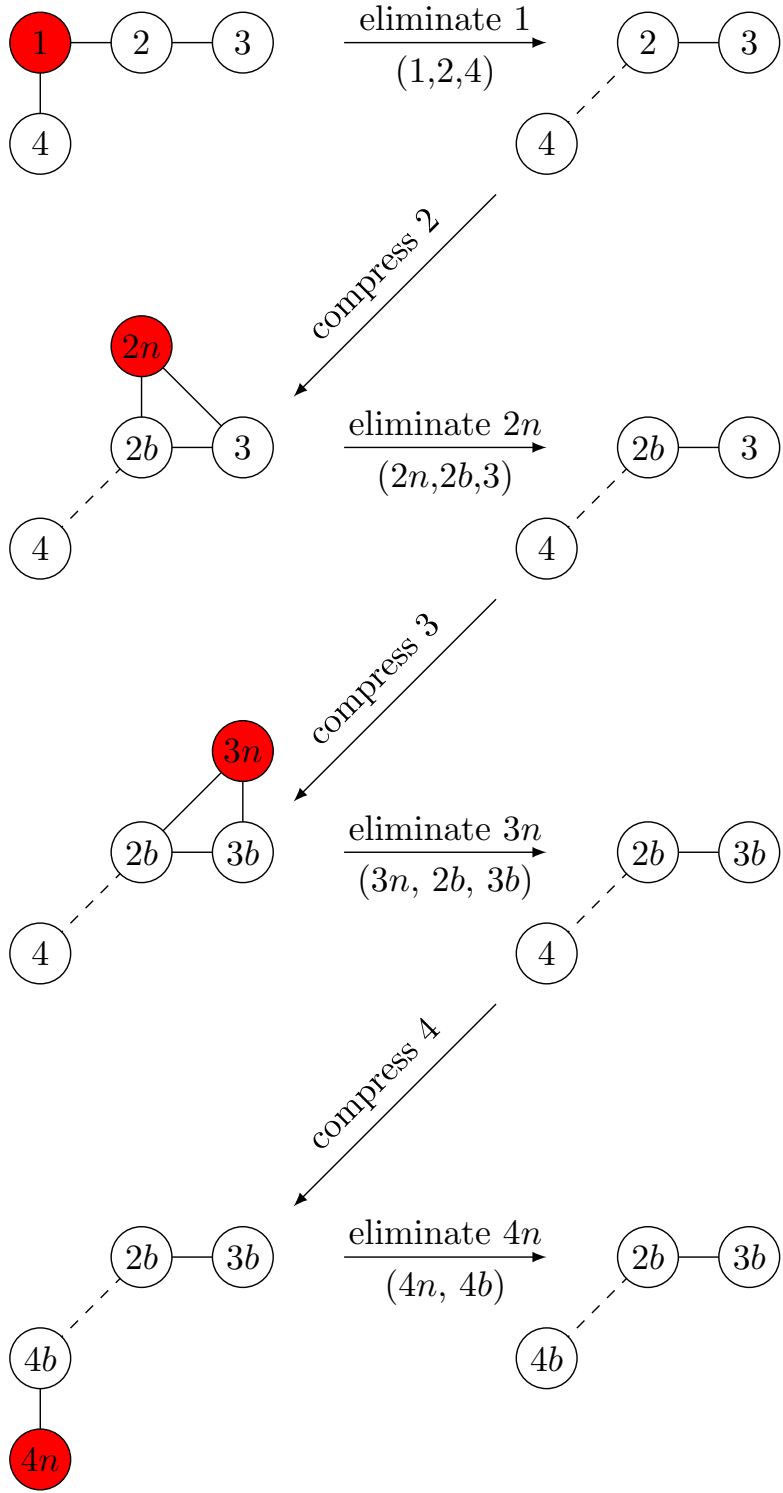

\centering
\resizebox{.8\textwidth}{!}{
\tikz{
\draw [] (0,0) -- (1,0);
\draw [] (1,0) -- (2,0);
\draw [] (0,-1) -- (0,0);
\draw [fill = red](0,0) circle (0.3);
\draw [fill = white](1,0) circle (0.3);
\draw [fill = white](2,0) circle (0.3);
\draw [fill = white](0,-1) circle (0.3);
\node [] at (0,0) {1};
\node [] at (1,0) {2};
\node [] at (2,0) {3};
\node [] at (0,-1) {4};

\begin{scope}[shift={(0.5,0.5)}]
\draw[-latex] (2.5,-0.5) -- (4.5,-0.5);
\node [above] at (3.5,-0.5) {eliminate 1};
\node [below] at (3.5,-0.5) {(1,2,4)};
\end{scope}

\begin{scope}[shift={(5,0)}]
\draw [] (1,0) -- (2,0);
\draw [dashed] (0,-1) -- (1,0);
\draw [fill = white] (1,0) circle (0.3);
\draw [fill = white] (2,0) circle (0.3);
\draw [fill = white] (0,-1) circle (0.3);
\node [] at (1,0) {2};
\node [] at (2,0) {3};
\node [] at (0,-1) {4};
\end{scope}

\begin{scope}[shift={(0,0)}]
\draw[-latex] (4.5,-1.5) -- (2.5,-3.5);
\node [above, rotate=45] at (3.5,-2.5) {compress 2};
\end{scope}

\begin{scope}[shift={(0,-4)}]
\draw [] (1,1) -- (2,0);
\draw [dashed] (0,-1) -- (1,0);
\draw [] (1,1) -- (1,0);
\draw [] (1,0) -- (2,0);
\draw [fill = red](1,1) circle (0.3);
\draw [fill = white] (1,0) circle (0.3);
\draw [fill = white] (2,0) circle (0.3);
\draw [fill = white] (0,-1) circle (0.3);
\node [] at (1,1) {$2n$};
\node [] at (1,0) {$2b$};
\node [] at (2,0) {3};
\node [] at (0,-1) {4};
\end{scope}

\begin{scope}[shift={(0.5,-3.5)}]
\draw[-latex] (2.5,-0.5) -- (4.5,-0.5);
\node [above] at (3.5,-0.5) {eliminate $2n$};
\node [below] at (3.5,-0.5) {($2n$,$2b$,3)};
\end{scope}

\begin{scope}[shift={(5,-4)}]
\draw [dashed] (0,-1) -- (1,0);
\draw [] (1,0) -- (2,0);
\draw [fill = white] (1,0) circle (0.3);
\draw [fill = white] (2,0) circle (0.3);
\draw [fill = white] (0,-1) circle (0.3);
\node [] at (1,0) {$2b$};
\node [] at (2,0) {3};
\node [] at (0,-1) {4};
\end{scope}

\begin{scope}[shift={(0,-4)}]
\draw[-latex] (4.5,-1.5) -- (2.5,-3.5);
\node [above, rotate=45] at (3.5,-2.5) {compress 3};
\end{scope}

\begin{scope}[shift={(0,-8)}]
\draw [] (2,1) -- (1,0);
\draw [dashed] (0,-1) -- (1,0);
\draw [] (2,1) -- (2,0);
\draw [] (1,0) -- (2,0);
\draw [fill = red](2,1) circle (0.3);
\draw [fill = white] (1,0) circle (0.3);
\draw [fill = white] (2,0) circle (0.3);
\draw [fill = white] (0,-1) circle (0.3);
\node [] at (2,1) {$3n$};
\node [] at (1,0) {$2b$};
\node [] at (2,0) {$3b$};
\node [] at (0,-1) {4};
\end{scope}

\begin{scope}[shift={(0.5,-7.5)}]
\draw[-latex] (2.5,-0.5) -- (4.5,-0.5);
\node [above] at (3.5,-0.5) {eliminate $3n$};
\node [below] at (3.5,-0.5) {($3n$, $2b$, $3b$)};
\end{scope}

\begin{scope}[shift={(5,-8)}]
\draw [dashed] (0,-1) -- (1,0);
\draw [] (1,0) -- (2,0);
\draw [fill = white] (1,0) circle (0.3);
\draw [fill = white] (2,0) circle (0.3);
\draw [fill = white] (0,-1) circle (0.3);
\node [] at (1,0) {$2b$};
\node [] at (2,0) {$3b$};
\node [] at (0,-1) {4};
\end{scope}

\begin{scope}[shift={(0,-8)}]
\draw[-latex] (4.5,-1.5) -- (2.5,-3.5);
\node [above, rotate=45] at (3.5,-2.5) {compress 4};
\end{scope}

\begin{scope}[shift={(0,-12)}]
\draw [dashed] (0,-1) -- (1,0);
\draw [] (1,0) -- (2,0);
\draw [] (0,-1) -- (0,-2);
\draw [fill = red](0,-2) circle (0.3);
\draw [fill = white] (1,0) circle (0.3);
\draw [fill = white] (2,0) circle (0.3);
\draw [fill = white] (0,-1) circle (0.3);
\node [] at (1,0) {$2b$};
\node [] at (2,0) {$3b$};
\node [] at (0,-1) {$4b$};
\node [] at (0,-2) {$4n$};
\end{scope}

\begin{scope}[shift={(0.5,-11.5)}]
\draw[-latex] (2.5,-0.5) -- (4.5,-0.5);
\node [above] at (3.5,-0.5) {eliminate $4n$};
\node [below] at (3.5,-0.5) {($4n$, $4b$)};
\end{scope}

\begin{scope}[shift={(5,-12)}]
\draw [dashed] (0,-1) -- (1,0);
\draw [] (1,0) -- (2,0);
\draw [fill = white] (1,0) circle (0.3);
\draw [fill = white] (2,0) circle (0.3);
\draw [fill = white] (0,-1) circle (0.3);
\node [] at (1,0) {$2b$};
\node [] at (2,0) {$3b$};
\node [] at (0,-1) {$4b$};
\end{scope}

}}
\caption{One compress-eliminate pass on an example matrix. The dashed
line $(2, 4)$ indicates an edge that was not present in the original graph. Arrows
representing elimination steps are also labeled with the nodes that are connected
after elimination (and therefore are connected in the computed part of $L$)}
\label{fig:one_lvl_ce_gr}
\end{figure}$\\$
\begin{figure}[H]
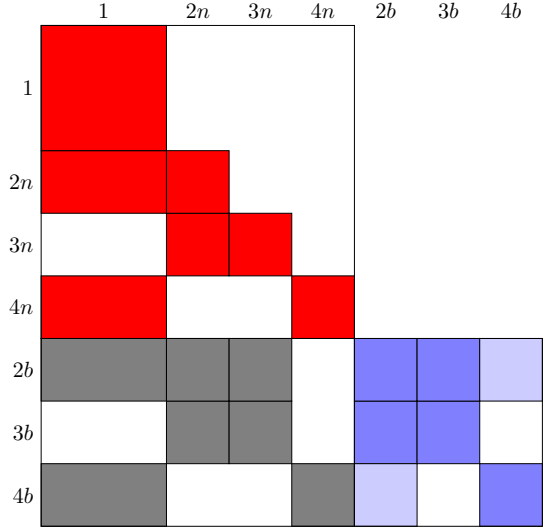

\centering
\resizebox{.56\textwidth}{!}{
\tikz{
\draw[fill = gray]  (0,0) rectangle (2,1);
\draw[fill = gray]  (0,2) rectangle (2,3);
\draw[fill = gray]  (2,1) rectangle (3,2);
\draw[fill = gray]  (2,2) rectangle (3,3);
\draw[fill = gray]  (3,1) rectangle (4,2);
\draw[fill = gray]  (3,2) rectangle (4,3);
\draw[fill = gray]  (4,0) rectangle (5,1);
\draw[] (0,0) rectangle (5,3);

\draw[fill = blue!50!white]  (5,1) rectangle (6,2);
\draw[fill = blue!50!white]  (5,2) rectangle (6,3);
\draw[fill = blue!50!white]  (6,1) rectangle (7,2);
\draw[fill = blue!50!white]  (6,2) rectangle (7,3);
\draw[fill = blue!50!white]  (7,0) rectangle (8,1);
\draw[fill = blue!20!white]  (5,0) rectangle (6,1);
\draw[fill = blue!20!white]  (7,2) rectangle (8,3);
\draw[] (5,0) rectangle (8,3);

\draw[fill = red]  (0,3) rectangle (2,4);
\draw[fill = red]  (0,5) rectangle (2,6);
\draw[fill = red]  (0,6) rectangle (2,8);
\draw[fill = red]  (2,4) rectangle (3,5);
\draw[fill = red]  (2,5) rectangle (3,6);
\draw[fill = red]  (3,4) rectangle (4,5);
\draw[fill = red]  (4,3) rectangle (5,4);
\draw[] (0,3) rectangle (5,8);

\node [left] at (0,0.5) {$4b$};
\node [left] at (0,1.5) {$3b$};
\node [left] at (0,2.5) {$2b$};
\node [left] at (0,3.5) {$4n$};
\node [left] at (0,4.5) {$3n$};
\node [left] at (0,5.5) {$2n$};
\node [left] at (0,7) {$1$};

\node [above] at (7.5,8) {$4b$};
\node [above] at (6.5,8) {$3b$};
\node [above] at (5.5,8) {$2b$};
\node [above] at (4.5,8) {$4n$};
\node [above] at (3.5,8) {$3n$};
\node [above] at (2.5,8) {$2n$};
\node [above] at (1,8) {$1$};

}}
\caption{Structure after one compress-eliminate pass on the example matrix.
The $L_1$ matrix, consisting of a near/near interaction matrix and near/far interaction
matrix (red and gray) have the same structure as $A$; the remaining
far/far interaction Schur complement $A_1$ (blue) contains fill blocks (light blue)
associated with edges in $A^2$ that were not in $A$.}
\label{fig:fig:one_lvl_ce_mat}
\end{figure}$\\$
}

\subsection{Sparsity of the factors}
Let us now study the sparsity of the factors $L_1$ and $A_1$. First define the block sparsity pattern of a block sparse matrix. For matrix~$A$ with $M$ block columns, $M$ block rows and block size $B \times B$ define $\bsp{A}{M}{B \times B}$ (block sparsity pattern of block-sparse matrix $A$) as a function
$$F_A(i,j) \rightarrow \{0,1\}, \forall i,j = 1\dots M$$ that returns $1$ if the block $A_{ij}$ is nonzero and $0$ otherwise. For simplicity, assume that in the first level elimination at each step we eliminate $(B - r)$ rows.
{\color{royalblue} Note that in this section we consider only the case with fixed rank $r$ and fixed block size $B$. Cases of adaptive choice of rank $r$ and non-uniform block size $B$ are the subjects for the future research.}
For the matrix~${\color{royalblue}\check{L}}_1$, according to equations~\eqref{eq:sm_L} and~\eqref{eq:big_L} we can see that
$$\bsp{{\color{royalblue}\check{L}}_1}{M}{B \times (B-r)} {\color{royalblue} \leqslant} \bsp{{\color{royalblue}A_0}}{M}{B \times B}.$$ The permutation matrix $P_1$ does not change the number of nonzero  {blocks in $L_1 = P_1{\color{royalblue}\check{L}}_1$.}
 {Denote the} number of nonzero blocks of matrix $A$  {as} $$\#\bsp{A}{M}{B \times B}.$$  Since  the permutation $P_1$ splits  the blocks of the matrix $L_1$ into ``eliminated'' and ``not eliminated'' parts,
\begin{equation}\#\bsp{L_1}{M}{(B-r)\times (B-r) \text{ or } r\times (B-r)} {\color{royalblue} \leqslant} 2\#\bsp{{\color{royalblue}\check{L}}_1}{M}{B \times (B-r)}{\color{royalblue} \leqslant}2 \#\bsp{{\color{royalblue}A_0}}{M}{B\times B}.
\label{eq:L}
\end{equation}
Unlike the matrix $L_1$, the sparsity of the matrix $A_1$ grows in a more complex way. Since the matrix~$A_1$ is an initial point for the next level it is very important to estimate its sparsity.
Thanks to \eqref{eq:A_st} we can see that\footnote{By $\bsp{A_1}{M}{r \times r} \leqslant \bsp{A^2}{M}{B \times B}$ we mean that {$F_{A_1}(i,j) \leqslant F_{{\color{royalblue}A_0}^2}(i,j), \forall i,j = 1\dots M.$}}
$$\bsp{A_1}{M}{r \times r} \leqslant \bsp{{\color{royalblue}A_0}^2}{M}{B \times B}.$$

The crucial part of the algorithm is that for the next level of elimination we  {combine}
the blocks of the matrix~$A_1$ into super-blocks of size $rJ \times rJ$ ($Jr \approx B$). Then we consider the new matrix with $M/J$ blocks (assume that $M$ is divisible by $J$). The number $J$ is the inner parameter of the algorithm. The new matrix is then considered as  an $MJ \times MJ$ matrix. (The new block size is $Jr$.\footnote{We can easily adapt for variable block size, but for simplicity we consider them to be equal.})
We can see that

\begin{equation}
\bsp{A_1}{(M/J)}{rJ\times {\color{royalblue}rJ}} \leqslant \bsp{{\color{royalblue}A_0}^2}{(M/J)}{BJ\times BJ}.
\label{eq:bsp}
\end{equation}

This is where the choice of the original ordering is required. {We consider the choice of permutation in Section~\ref{sec:complan}.}

\subsection{Multilevel step}
Consider the matrix $A_1$ with $Jr \times Jr$ blocks and use the compression-elimination procedure to find corresponding matrices $Q_2$, $L_2$ and $P_2$:
$$P_2Q_2A_1Q_2^{\top}P^{\top}_2 \approx L_2 L^{\top}_2 + \begin{bmatrix}
0_{N_{L2}\times N_{L2}}& 0 \\
0&A_2
\end{bmatrix}.$$
We introduce
$$ \breve{Q}_2 = \begin{bmatrix}
I_{N_{L1} }& 0 \\
0&Q_2
\end{bmatrix},\quad \breve{P}_2 = \begin{bmatrix}
I_{N_{L1}}& 0 \\
0&P_2
\end{bmatrix},  \quad\text{and } \quad{\color{royalblue}\breve{L}}_2 = \begin{bmatrix}
0_{N_{L1} \times N_{L2}}\\
L_2
\end{bmatrix}.$$
then we have two-level approximate factorization
\begin{equation}\breve{P}_2\breve{Q}_2 \breve{P}_1 \breve{Q}_1 \, {\color{royalblue}A_0}\, \breve{Q}_1^{\top}\breve{P}_1^{\top} \breve{Q}_2^{\top}\breve{P}_2^{\top} \approx \begin{bmatrix} L_1& {\color{royalblue}\breve{L}}_2\end{bmatrix}\begin{bmatrix} L_1& {\color{royalblue}\breve{L}}_2\end{bmatrix}^{\top} +\begin{bmatrix}
 0&0 \\
 0 & A_2
\end{bmatrix},
\end{equation}\label{eq:scnd_lvl}
When the size is small we do the Cholesky factorization {of the remainder}.
The algorithm is summarized in the next proposition.
\begin{proposition}
\label{prop:CE_L}
The compression-elimination (CE) algorithm leads to the approximate factorization
\begin{equation}\label{eq:ce}A \approx Q^{\top}LL^{\top}Q,\end{equation} where $Q$ is a unitary matrix equal to the multiplication of block-diagonal and permutation matrices
{\color{royalblue}$$Q = \left( \prod_{j=1}^{K}\breve{P}_j \breve{Q}_j \right) P,$$}
L is a sparse block {lower-triangular} matrix
$$L = \begin{bmatrix}
\begin{bmatrix}
\\ L_1\\ \\ \\
\end{bmatrix}
&\begin{matrix}
0\\
\begin{bmatrix}
 \\L_2 \\ \\
\end{bmatrix}
\end{matrix}
&\dots &
\begin{matrix}
0\\\vdots \\0\\
\begin{bmatrix}
L_{K+1}
\end{bmatrix}
\end{matrix}
\end{bmatrix} = \begin{bmatrix} L_1 & \breve{L}_2 & \dots & \breve{L}_{K+1} \end{bmatrix},$$
where block $L_i$ has
$$\#L_j = 2 \#\bsp{A_{j-1}}{M/J^{j-1}}{rJ \times rJ}$$ nonzero blocks, and the total number of nonzero blocks in factor $L$ is
\begin{equation}
\label{eq:l_nnz}
\begin{aligned}
\#L {\color{royalblue} \leqslant} 2\#\bsp{{\color{royalblue}A_0}}{M}{B \times B} + 2 \left( \sum_{j=1}^{K-1} \#\bsp{{\color{royalblue}A_0}^{2^{j}}}{M/J^{j}}{BJ^{j}\times BJ^{j}} \right) + \\
+ \frac{1}{2}(M/J^K)^2 r^2.
\end{aligned}
\end{equation}
$K$ is the number of levels in CE algorithm, $B$ is the block size.}
{ Time complexity for the approximate CE factorization is  $\mathcal{O}(B^3(\#L_{K+1}-\#L_1) + B^2\#L)$.
Memory complexity is {\color{royalblue}$\mathcal{O}(B(B-r)\#L + NB)$.}
We denote this approximate factorization \eqref{eq:ce} by \emph{CE factorization} of matrix $A$.
\end{proposition}
\begin{proof}
After $K$ levels,
\begin{equation}\label{eq:ce_full}
	\begin{aligned}
		\breve{P}_K\breve{Q}_K \dots  \breve{P}_1 \breve{Q}_1 \, {\color{royalblue}A_0}\, \breve{Q}_1^{\top}\breve{P}_1^{\top} \dots  \breve{Q}_K^{\top}\breve{P}_K^{\top} \approx \\
		 \begin{bmatrix} L_1& {\color{royalblue}\breve{L}}_2 &\dots &{\color{royalblue}\breve{L}}_K \end{bmatrix}\begin{bmatrix} L_1^{\top}\\ {\color{royalblue}\breve{L}}_2^{\top} \\\vdots \\ {\color{royalblue}\breve{L}}_K^{\top} \end{bmatrix}+\begin{bmatrix}
0_{N_*\times N_{*}} &0 \\
0 & A_K
\end{bmatrix}
\end{aligned},\end{equation}
where $N_{*} = \sum_{j=1}^{K}N_{L_j}.$ {\color{royalblue}Using equation~\eqref{eq:init_perm},} denote
\begin{equation}\label{eq:q} Q = \left ( \prod_{j=1}^{K}\breve{P}_j \breve{Q}_j\right) P,\end{equation}
The remainder is factorized exactly as
$$A_K = L_{K+1}L_{K+1}^{\top}.$$ Denote
$$L = \begin{bmatrix}
\begin{bmatrix}
\\ L_1\\ \\ \\
\end{bmatrix}
&\begin{matrix}
0\\
\begin{bmatrix}
 \\L_2 \\ \\
\end{bmatrix}
\end{matrix}
&\dots &
\begin{matrix}
0\\\vdots \\0\\
\begin{bmatrix}
L_{K+1}
\end{bmatrix}
\end{matrix}
\end{bmatrix} = \begin{bmatrix} L_1 & \breve{L}_2 & \dots & \breve{L}_{K+1} \end{bmatrix},$$
where $L_j$ are rectangular block-triangular matrices. {
Using introduced notation and equation~\eqref{eq:ce_full} we arrive at~\eqref{eq:ce}.
The number $\#L$ of nonzero blocks in the factor $L$ is}
\begin{equation}\#L = \sum_{j=1}^{K+1} \#L_j.\label{eq:suml}\end{equation}
Using \eqref{eq:L},
$$\#L_1 {\color{royalblue} \leqslant} 2 \#\bsp{{\color{royalblue}A_0}}{M/J}{B\times B},$$
for $j = 2 \dots K$:
\begin{equation}
\label{eq:aj}
\#L_j {\color{royalblue} \leqslant} 2 \#\bsp{A_{j-1}}{M/J^{j-1}}{rJ \times rJ},
\end{equation}
and for $j = K+1$:
$$\#L_{K+1} {\color{royalblue} \leqslant}  \frac{1}{2}(M/J^K)^2 r^2.$$
Using equation \eqref{eq:bsp} we can see that
$$\bsp{A_K}{(M/J^K)}{r \times r} \leqslant \bsp{A_{K-1}^2}{(M/J^{K})}{rJ \times rJ} \leqslant \dots $$ $$\dots  \leqslant \bsp{{\color{royalblue}A_0}^{2^K}}{(M/J^{K})}{BJ^{K} \times BJ^{K}}.$$
Finally{, we obtain equation~\eqref{eq:l_nnz}.

In the CE algorithm we need to store $Q$ and $L$ factors.
According to~\eqref{eq:q}, the factor $Q$ is a product of $K$ permutation matrices $\breve{P}_1,\dots,\breve{P}_K$  and $K$ block-diagonal matrices $\breve{Q}_1,\dots,\breve{Q}_K$ with block size $B$ {\color{royalblue}and number of blocks gets smaller at each step by a factor of $J$.  Since
\begin{equation} \sum_{k=0}^{\infty} \frac{N}{J^k}B = NB \frac{J}{J-1},
	\label{eq:req01}
\end{equation} the matrix $Q$ can be stored using $\mathcal{O}(NB)$ memory.}
The memory requirement for the matrix $L$ is $\mathcal{O}(B(B-r)\#L)$.
Therefore, the factorization requires  {\color{royalblue}$\mathcal{O}(B(B-r)\#L + NB)$} memory.

Now let us estimate the number of operations in the construction of the CE factorization.
The CE factorization is done in $K$ sweeps, each sweep consists of several compression and elimination steps. Let us compute the computational cost of the $i$-th sweep.
Compression step requires the computation of the SVD factorizations for the the far blocks and multiplication
of the block rows and columns by the left SVD factor.
The matrix $A_i$ has  $(\#L_{i+2}-\#L_{i+1})$ far blocks of size $B\times B$, thus the SVD factorizations require $\mathcal{O}(B^3(\#L_{i+2}-\#L_{i+1}))$ operations.
The multiplication of the matrix $A_i$ by the unitary block-diagonal matrix with block size $B$ (since the matrix $A_i$ has $\#L_{i+1}$ number of blocks) requires $\mathcal{O}(B^2\#L_{i+1})$ operations.
Elimination procedure (using the block Cholesky factorization~\eqref{eq:bl_chol}) for all rows of the matrix $A_i$ requires  $\mathcal{O}((B-r)^2\#L_{i+1})$ operations.
Thus, $i$-th sweep of the CE algorithm costs $\mathcal{O}\left((B)^3(\#L_{i+1}-\#L_{i}) + B^2\#L_{i+1}\right)$.
Total computational cost of $K$ sweeps of the CE algorithm is
$$t = \sum_{i = 0}^{K} \left(\mathcal{O}(B^3(\#L_{i+2}-\#L_{i+1})+ B^2\#L_{i+1})\right) = $$
$$\mathcal{O}\left(B^3(\#L_{K+1}-\#L_1)\right) + B^2 \sum_{i = 0}^{K} \mathcal{O}(\#L_{i+1}),$$
taking into account equation \eqref{eq:suml},
$$t = \mathcal{O}\left(B^3(\#L_{K+1}-\#L_1) + B^2\#L\right).$$}
\end{proof}
\begin{proposition}
The system $Ax = b$ with a precomputed CE-factorization of the matrix $A$ ($A = Q^{\top}LL^{\top}Q$) can be solved in $\mathcal{O}((B-r)B\#L + NB)$ operations.
\end{proposition}
\begin{proof}
Since
$$Q^{\top}LL^{\top}Q x = b,$$
then
$$ x = Q^{\top}L^{-\top}L^{-1}Qb.$$
We need to compute the matrix $Q$ and the matrix $Q^{\top} $ by vector products. According to~\eqref{eq:q}, the matrices $Q$ and $Q^{\top}$ are the product of $K$ permutation matrices $\breve{P}_1,\dots,\breve{P}_K$ and $K$ block-diagonal matrices $\breve{Q}_1\dots \breve{Q}_K$ with block size $B$
{\color{royalblue}
and number of blocks gets smaller at each step by a
factor of $J$. Taking into account equation~\eqref{eq:req01},
the total computational cost of the matrix $Q$ and the matrix $Q^{\top} $ by vector products is $\mathcal{O}(NB)$ operations.
}

Since $L$ is a block-triangular sparse matrix with $\#L$ nonzero $(B-r)\times B$ blocks,
the solution of the system with matrix $L$ requires  $\mathcal{O}((B-r)B\#L)$ operations. Similarly, the solution of the system with matrix $L^{\top}$ requires $\mathcal{O}((B-r)B\#L)$ operations.
In total, we need {\color{royalblue}$\mathcal{O}((B-r)B\#L + NB)$} operations to solve the system $Q^{\top}LL^{\top}Q x = b$.
\end{proof}

\subsection{Pseudo code} Here we present scheme of the CE factorization algorithm.
\begin{algorithm}[H]
\SetKwProg{D}{Input:}{}{}
\SetKwProg{ol}{Level elimination: (see the Algorithm \ref{al:ol})}{}{}
\SetKwProg{I}{Initialization:}{}{}
\SetKwFunction{Out}{Output:}
\SetKwProg{fact}{Factorization:}{}{}
\SetKwProg{bs}{Backward step:}{}{}
\SetKwProg{O}{Output:}{}{}
\SetKw{ret}{return:}
\caption{CE algorithm}
\D{}{
{\color{royalblue}$A \in \mathbb{R}^{N\times N}$ - sparse matrix,\\
$A_0 = P A P^{\top}$, $A_0 \in\mathbb{R}^{MB \times MB}$ - block sparse matrix,\\}
$K$ - number of levels\\
$\varepsilon$ or $r$ - rank revealing parameter\\
}
\fact{}{
\For{j = 1 \dots  K}{
    $M_j = \frac{N}{BJ^{(j-1)}}$ - number of blocks on current level\\
    \ol{}{
    $\bold{Input:}$ $A_{j-1}$, $M_j$, $r$ or $\varepsilon$\\
    $\bold{Output:}$ ${\color{royalblue}\check{A}}_j$, ${\color{royalblue}\check{L}}_{j}$, $Q_j$}
    $L_j =  P_j^{\top}{\color{royalblue}\check{L}}_{j},$\\
    $A_j = P_j{\color{royalblue}\check{A}}_jP_j^{\top}$ - remainder.\\
	{\color{royalblue}
	Extend $Q_j$ by identity matrix, obtain $\breve{Q}_j$\\
	Extend $P_j$ by identity matrix, obtain $\breve{P}_j$\\
	Extend $L_j$ by zero matrix, obtain $\breve{L}_j$\\
	}
}
Factorize $A_K = L_{K+1}L_{K+1}^{\top},$\\
$Q = \left(\prod_{j=1}^{K} {\color{royalblue}\breve{P}}_j{\color{royalblue}\breve{Q}}_j\right){\color{royalblue}P},$\\
$L = \begin{bmatrix} L_1 &{\color{royalblue}\breve{L}}_2& \dots  & {\color{royalblue}\breve{L}}_{K+1} \end{bmatrix}.$ \\ 
}
\O{}{ $L,Q$ (where $A \approx Q^{\top}LL^{\top}Q$).}
\end{algorithm}

~

\begin{algorithm}[H]
\SetKwProg{D}{Input:}{}{}
\SetKwProg{ol}{{\color{royalblue}Elimination of $j$-th level}}{}{}
\SetKwProg{C}{Compression step:}{}{}
\SetKwProg{E}{Elimination step:}{}{}
\caption{One level elimination}
\ol{}{
    \D{}{$A_{j-1}$, $M_j$, $\varepsilon$ or $r$\\}
	 {\color{royalblue}$\widehat{A}_{0} = A_{j-1}$}\\
    \For{i = 1 \dots  $M_j$}{
        \C{}{
            $\widetilde{R}_i=\begin{bmatrix}D_i&C_i&F_i&0\end{bmatrix}P^{\text{col}}_i$ - $i$-th block row {\color{royalblue} of $\widehat{A}_{i-1}$},\\
            $F_i \approx \widetilde{U}_i\begin{bmatrix} \widehat{F}_i\\0\end{bmatrix},$ $\widetilde{U}_i$ - unitary matrix, $\widehat{F}_i \in \mathbb{R}^{r\times BN^F_i}$ ,\\
            $\widetilde{R}_{i*} = \begin{bmatrix}\widetilde{U}_i^{\top}D_i\widetilde{U}_i&\widetilde{U}_i^{\top}C_i&\begin{bmatrix}\widehat{F}_i\\0\end{bmatrix}&0\end{bmatrix}P^{\text{col}}_i$. \\
        }
        \E{}{
            $\widetilde{S}_i = \begin{bmatrix}0_{r \times r}&0 \\ 0 & I_{(B-r)}\end{bmatrix}\widetilde{R}_{i*}$,\\
            Eliminate $\widetilde{S}_i$ using block Cholesky formula:\\
            $ {\color{royalblue}\widehat{A}}_{i-1} = \widetilde{L}_i\widetilde{L}_i^{\top} + \widetilde{A}_i, $\\
            Add $\widetilde{L}_i$ into matrix $\check{L}_j.$
         }
    }
${\color{royalblue}Q_j} = \mathrm{diag}(\widetilde{U}_1, \dots  ,\widetilde{U}_{M_j})$.\\
{\color{royalblue}$\check{A}_j = \widetilde{A}_{M_j}$}\\
$\bold{Output:}$ $ {\color{royalblue}\check{A}_j}$, {\color{royalblue}$\check{L}_j$}, {\color{royalblue}$Q_j$}
}
\label{al:ol}
\end{algorithm}

{
\subsection{CE complexity based on the graph properties}
\label{sec:complan}
Consider the undirected graph $\mathcal{G}_{\color{royalblue}A_0}$ associated with the block sparsity pattern of the symmetric matrix ${\color{royalblue}A_0}$: $i$-th graph node corresponds to $i$-th block rows and columns, if the block $(i,j)$ is nonzero then $i$-th and $j$-th graph nodes are connected by an edge. In this section we clarify the connection between properties of the graph $\mathcal{G}_{\color{royalblue}A_0}$ and the complexity of the CE algorithm.

For example, consider the matrix ${\color{royalblue}A_0}$ with the graph shown in Figure~\ref{fig:sub01}. This matrix can arise from the the 2D elliptic equation discretized on a uniform square grid in $[0,1]^2$ by finite-difference scheme with $9$-point stencil.
Note that each graph node is connected only with a fixed number of nearest spatial neighbors (we call this property \emph{the graph locality}.)

According to \eqref{eq:bsp}, one sweep of the CE algorithm leads to squaring of the block sparsity pattern of the remainder. 
Squaring of the matrix ${\color{royalblue}A_0}$ spoils the locality of the graph $\mathcal{G}_{{\color{royalblue}A_0}}$: if two nodes were connected to the same node in the graph $\mathcal{G}_{\color{royalblue}A_0}$, they are connected in the graph  $\mathcal{G}_{{\color{royalblue}A_0}^2}$
(see the graph $\mathcal{G}_{{\color{royalblue}A_0}^2}$ in Figure~\ref{fig:sub02}). We call this step ``squaring''
($\mathcal{G}_{{\color{royalblue}A_0}^2} = \mathbf{Sq}(\mathcal{G}_{\color{royalblue}A_0})$).
Note that the squaring procedure significantly increases the number of edges in the graph $\mathcal{G}_{{\color{royalblue}A_0}^2}$.

The next step of the CE factorization joins block rows and columns into super-blocks by groups of $J$.
We obtain the matrix $A_1$ from equation~\eqref{eq:rem1}.
For the graph $\mathcal{G}_{{\color{royalblue}A_0}^2}$
it means joining nodes into super-nodes by groups of $J$. We call this step ``coarsening''
$\mathcal{G}_{A_1} = \mathbf{Coars}(\mathcal{G}_{{\color{royalblue}A_0}^2})$.
Coarsened graph $\mathcal{G}_{A_1}$ has better locality than the graph $\mathcal{G}_{{\color{royalblue}A_0}^2}$.

In this particular example the graph structure of the initial matrix ${\color{royalblue}A_0}$ and of the squared-coarsened matrix $A_1$ is very similar, see Figure.~\ref{fig:sub03}.
During the CE algorithm we obtain the graphs $ \mathcal{G}_{{\color{royalblue}A_0}}, \mathcal{G}_{A_1}, \dots, \mathcal{G}_{A_K}$, where $\mathcal{G}_{A_i} = \mathbf{Coars}(\mathbf{Sq}(\mathcal{G}_{A_{i-1}})), $ $ \forall i = 1,\dots,K$.

\begin{figure}[H]
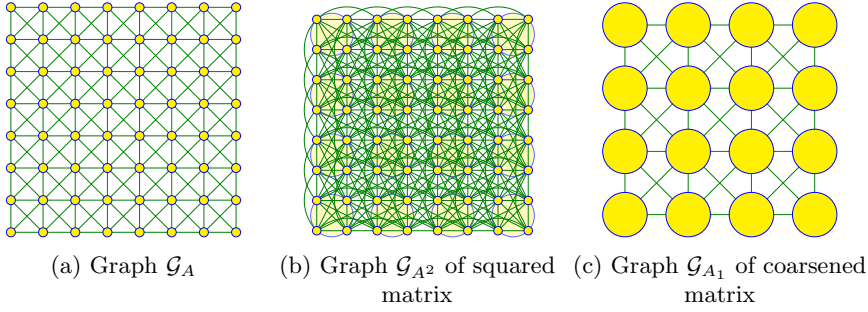

\centering
\begin{subfigure}[t]{.3\textwidth}
  \centering
  \resizebox{.8\textwidth}{!}{\tikz{
    \draw[step=1cm,black!50!green] (0,0) grid (7,7);
    \foreach \i in {0,...,6}{
        \foreach \j in {0,...,6}{
            \draw[line width=.7pt,black!50!green] (\i,\j) -- (\i+1, \j+1);
            \draw[line width=.7pt,black!50!green] (\i,\j+1) -- (\i+1, \j);
        }
    }
    \foreach \i in {0,...,7}{
        \foreach \j in {0,...,7}{
            \draw[line width=.7pt,blue,fill=yellow] (\i,\j) circle (4pt);

        }
    }
}}
  \caption{\centering Graph $\mathcal{G}_A$}
  \label{fig:sub01}
\end{subfigure}%
\begin{subfigure}[t]{.3\textwidth}
  \centering
  \resizebox{.8\textwidth}{!}{\tikz{
    \foreach \i in {0,...,3}{
        \foreach \j in {0,...,3}{
            \draw[line width=.7pt,blue!60!white,fill=yellow!30!white] (2*\i+0.5,2*\j+0.5) circle (20pt);
        }
    }
    \draw[step=1cm,black!50!green] (0,0) grid (7,7);
    \foreach \i in {0,...,6}{
        \foreach \j in {0,...,6}{
            \draw[line width=.7pt,black!50!green] (\i,\j) -- (\i+1, \j+1);
            \draw[line width=.7pt,black!50!green] (\i,\j+1) -- (\i+1, \j);
        }
    }
    \foreach \i in {0,...,7}{
        \foreach \j in {2,...,7}{
            \draw [line width=.7pt,black!50!green](\i,\j) arc (135:225:1.41);
        }
    }
    \foreach \i in {2,...,7}{
        \foreach \j in {0,...,7}{
            \draw [line width=.7pt,black!50!green](\i,\j) arc (45:135:1.41);
        }
    }

    \foreach \i in {0,...,5}{
        \foreach \j in {2,...,7}{
        }
    }
    \foreach \i in {2,...,7}{
        \foreach \j in {2,...,7}{
        }
    }
    \foreach \i in {1,...,7}{
        \foreach \j in {2,...,7}{
            \draw [line width=.7pt,black!50!green](\i,\j) -- (\i-1,\j-2);
        }
    }
    \foreach \i in {2,...,7}{
        \foreach \j in {1,...,7}{
            \draw [line width=.7pt,black!50!green](\i,\j) -- (\i-2,\j-1);
        }
    }
    \foreach \i in {1,...,7}{
        \foreach \j in {0,...,5}{
            \draw [line width=.7pt,black!50!green](\i,\j) --  (\i-1,\j+2);;
        }
    }
    \foreach \i in {2,...,7}{
        \foreach \j in {0,...,6}{
            \draw [line width=.7pt,black!50!green](\i,\j) -- (\i-2,\j+1);
        }
    }
    \foreach \i in {0,...,7}{
        \foreach \j in {0,...,7}{
            \draw[line width=.7pt,blue,fill=yellow] (\i,\j) circle (4pt);
        }
    }

}}
  \caption{\centering Graph~$\mathcal{G}_{A^2}$ of squared matrix}
  \label{fig:sub02}
\end{subfigure}
\begin{subfigure}[t]{.3\textwidth}
  \centering
  \resizebox{.8\textwidth}{!}{\tikz{
    \draw[step=2cm,black!50!green,shift={(0.5,0.5)}] (0,0) grid (6,6);
         \foreach \i in {0,...,2}{
        \foreach \j in {0,...,2}{
            \draw[black!50!green] (2*\i+0.5,2*\j+0.5) -- (2*\i+2.5,2*\j+2.5);
        }
    }
     \foreach \i in {0,...,2}{
        \foreach \j in {0,...,2}{
            \draw[black!50!green] (2*\i+2.5,2*\j+0.5) -- (2*\i+0.5,2*\j+2.5);
        }
    }
    \foreach \i in {0,...,3}{
        \foreach \j in {0,...,3}{
            \draw[line width=.7pt,blue,fill=yellow] (2*\i+0.5,2*\j+0.5) circle (20pt);
        }
    }
}}
  \caption{\centering Graph~$\mathcal{G}_{A_1}$ of coarsened matrix}
  \label{fig:sub03}
\end{subfigure}%
\caption{Example of squaring-coarsening procedure for graph~$\mathcal{G}_{A}$}
\label{fig:ex1}
\end{figure}$\\$

Note that the number of edges in the graphs $ \mathcal{G}_{{\color{royalblue}A_0}}, \mathcal{G}_{A_1}, \dots, \mathcal{G}_{A_K}$ is related to the number $\#L$, that we need to estimate, see Proposition~\ref{prop:CE_L}.
\begin{proposition}
\label{prop:gl}
The number $\#L$ of nonzero blocks in the factor $L$ is equal to the total number of edges of all graphs  $ \mathcal{G}_{{\color{royalblue}A_0}}, \mathcal{G}_{A_1}, \dots, \mathcal{G}_{A_K}$:
$$\#L { \color{royalblue}\leqslant \sum_{i=0}^{K} \mathbf{Edge\_num}(\mathcal{G}_{A_i})}.$$
\end{proposition}
\begin{proof}
The number of edges of the graph $\mathcal{G}_{A_i}$ is equal to the number of nonzero blocks in the matrix $A_i$ $\forall i\in 0 \dots K$ by the definition.
Taking into account equation~\eqref{eq:l_nnz} we obtain:
\begin{equation}
\label{eq:l1}
\begin{aligned}
\#L { \color{royalblue}\leqslant } 2\#\bsp{{\color{royalblue}A_0}}{M}{B \times B} + 2 \left( \sum_{j=1}^{K-1} \#\bsp{A^{2^{j}}}{M/J^{j}}{BJ^{j}\times BJ^{j}} \right) + \\
+ \frac{1}{2}(M/J^K)^2 r^2 = \mathbf{Edge\_num}(\mathcal{G}_{{\color{royalblue}A_0}})+\sum_{i=1}^{K} \mathbf{Edge\_num}(\mathcal{G}_{A_i}).
\end{aligned}
\end{equation}
\end{proof}

Thus, the minimization of the number $\#L$ is equivalent to the minimization of the total number of edges in graphs $\mathcal{G}_{{\color{royalblue}A_0}}, \mathcal{G}_{A_1}, \dots, \mathcal{G}_{A_K}$.

The coarsening procedure is the key to the reduction of the number of edges. Note that the coarsening procedure for all graphs $\mathcal{G}_{{\color{royalblue}A_0}}, \mathcal{G}_{A_1}, \dots, \mathcal{G}_{A_K}$ can be defined by a single initial permutation $P$ {\color{royalblue} from equation~\eqref{eq:init_perm}} and the number of nodes~$J$ that we join together (neighbors in the permutation~$P$ are joined by the groups of~$J$). Also note that, with known $J$, the coarsening procedure determines the initial permutation $P$. Thus, to minimize $\#L$, instead of optimizing the initial permutation~$P$ we can optimize the coarsening procedure, which can be much more intuitive.

Consider the example in Figure~\ref{fig:ex1}. The following proposition shows that for this example a good coarsening exists (and is very simple).

\begin{proposition}
\label{prop:2d}
Let the graph $\mathcal{G}_{{\color{royalblue}A_0}}$ be defined by a tensor product grid in $\mathbb{R}^2$, (grid points correspond to nodes of the graph $\mathcal{G}_{{\color{royalblue}A_0}}$, edges of the grid correspond to edges of the graph $\mathcal{G}_{{\color{royalblue}A_0}}$), like in the example in Figure~\ref{fig:ex1}.
If the coarsening procedure joins the closest nodes and each super-node has at least two nodes in each direction,
then each graph $ \mathcal{G}_{{\color{royalblue}A_0}}, \mathcal{G}_{A_1}, \dots, \mathcal{G}_{A_K}$ has $\mathcal{O}(N)$ edges.

\end{proposition}
\begin{proof}
Consider the graph $\mathcal{G}_{{\color{royalblue}A_0}}$ based on the tensor product grid in $\mathbb{R}^2$, let each node $n$ of this graph have index $(i,j)$.
Let $s(e)$ be the length of the edge $e$ that connects nodes $n_1$ and $n_2$ with indices $(i_1,j_1)$ and $(i_2,j_2)$ if
$$s(e) = \max(|i_1-i_2|, |j_1-j_2|).$$
Let $\tilde{s}(\mathcal{G}_{{\color{royalblue}A_0}})$ be the  maximum edge length of the graph $\mathcal{G}_{{\color{royalblue}A_0}}$:
$$\tilde{s}(\mathcal{G}_{{\color{royalblue}A_0}}) = \max_{e_i\in\mathcal{G}_{{\color{royalblue}A_0}}}s(e_i).$$

For the graph  $\mathcal{G}_{A_i}$ of the matrix $A_i$, $\forall i\in 0\dots (K-1)$, $\tilde{s}(\mathcal{G}_{A_i}) = 1$. By the squaring procedure for the graph $\mathcal{G}_{A_i^2} = \mathbf{Sq}(\mathcal{G}_{A_l})$,  $\tilde{s}(\mathcal{G}_{A_i^2}) = 2$.
Let us prove that after the coarsening procedure described in the proposition hypothesis
($\mathcal{G}_{A_{i+1}}= \mathbf{Coars}(\mathcal{G}_{A_i^2})$) we have
$\tilde{s}(\mathcal{G}_{A_{i+1}}) = 1$. If $\tilde{s}(\mathcal{G}_{A_{i+1}}) > 1$, then, since each super-node has at least two nodes in each direction,
$\tilde{s}(\mathcal{G}_{A_i^2}) > 3$, this is a contradiction. Thus, $\tilde{s}(\mathcal{G}_{A_{i+1}}) = 1$.

If $\tilde{s}(\mathcal{G}_{A_{i+1}}) = 1$, then each node has a constant number of edges. Thus, the number of edges in the graph $\mathcal{G}_{A_i}, \forall i \in 0\dots (K-1)$ is $\mathcal{O}(N)$.
\end{proof}
\begin{corollary}
Proposition~\ref{prop:2d} is also true for $\mathbb{R}^d, d>2$.
\end{corollary}
\begin{proof}
Analogously to Proposition~\ref{prop:gl}.
\end{proof}
\begin{corollary}
\label{cor:comp}
Let the matrix $A_{\mathbf{tpg}}$ have the graph based on the tensor product grid in $\mathbb{R}^d$.
The CE factorization of the matrix $A_{\mathbf{tpg}}$ requires $\mathcal{O}(B^3N + B^2NK)$ operations and $\mathcal{O}(B(B-r)NK)$ memory.
\end{corollary}
\begin{proof}
According to Proposition~\ref{prop:2d}, each graph $ \mathcal{G}_{{\color{royalblue}A_0}}, \mathcal{G}_{A_1}, \dots, \mathcal{G}_{A_K}$ has $\mathcal{O}(N)$ edges. Thus, by Proposition~\ref{prop:gl}, $\#L = \mathcal{O}(NK)$, $(\#L_{K+1} - \#L_1) = \mathcal{O}(N)$. Therefore, by Proposition~\ref{prop:CE_L}, memory complexity of the CE factorization of the matrix $A_{\mathbf{tpg}}$ is the following:
{\color{royalblue}
$$\mathbf{mem} = \mathcal{O}(B(B - r)\#L + NB)  = \mathcal{O}(B(B - r)NK + NB) = \mathcal{O}(B(B - r) NK),$$}
and the time complexity:
$$\mathbf{t} = \mathcal{O}(B^3(\#L_{K+1} - \#L_1) + B^2\#L)  = \mathcal{O}(B^3N + B^2NK).$$
\end{proof}

Note, that graph partitioning algorithms \cite{rajamanickam-hg_part-2012,devine-hg_part-2006} and, in particular,  the nested dissection procedure \cite{george-1973-nested} can be used for grouping the closest blocks during the coarsening procedure.

Let us now discuss the choice of the coarsening procedure in the case of geometric graphs. Consider the $k$-nearest neighborhood graph \cite{miller-gr_negh-1991}.
This graph has a geometric interpretation, where each node has at most  $k$ edges and has a sphere that contains all connected nodes and only them. For example, the graph $\mathcal{G}_{{\color{royalblue}A_0}}$ in Figure~\ref{fig:sub02} is a $8$-nearest neighborhood graph.
 We propose the following hypothesis.
\begin{hypothesis}
\label{hyp:knn}
Let graph $\mathcal{G}_{\color{royalblue}A_0}$ be a $k$-nearest neighborhood graph; let
 $h_{\min}$ and $h_{\max}$ be the minimum and maximum edge lengths; let $B$ be the maximum block size.
Let graphs $\mathcal{G}_{A_1}$, \dots $\mathcal{G}_{A_K}$ be obtained in $K$ squaring-coarsening steps.
Then there exists a coarsening procedure and a number
$$k_1 = k_1 \left(\frac{h_{\max}}{h_{\min}},d,B\right),$$
such that the graphs $\mathcal{G}_{A_1}$, \dots $\mathcal{G}_{A_K}$ are
$k_1$-nearest neighborhood graphs.
\end{hypothesis}
\begin{remark}
The $k$-nearest neighborhood graph has $\mathcal{O}(N)$ edges if $k$ does not depend on $N$.
Thus,  Hypothesis~\ref{hyp:knn} can be reformulated in the following way: if $\mathcal{G}_{\color{royalblue}A_0}$ is a $k$-nearest neighborhood graph, then there exists a coarsening procedure such that each of the graphs $\mathcal{G}_{A_1}$, \dots $\mathcal{G}_{A_K}$ has $\mathcal{O}(N)$ edges.
Thus, analogously to Corollary~\ref{cor:comp},
the CE factorization of the matrix with $k$-nearest neighborhood graph requires $\mathcal{O}(B^3N + B^2NK)$ operations and $\mathcal{O}(B(B − r)NK)$ memory.
\end{remark}
}

\section{Numerical experiments}

We compare the following solvers: CE($\varepsilon$) with backward step as direct solver, CE($\varepsilon$) and CE($r$) factorizations as a preconditioners for iterative symmetric solver MINRES, direct symmetric solver from the package CHOLMOD and MINRES with ILUt preconditioner.
The required accuracy is {set} $10^{-10}$. {Note that for the solution of PDEs this accuracy is much less than the discretization accuracy, so often a much larger threshold is required.}

For the tests we consider the system obtained from an equidistant cubic discretization of 3D diffusion equation.
\begin{equation}
\left.\begin{matrix}
 \mathrm{div}\left( \, k(x) \, \mathrm{grad}\, u(x)\right) = f(x), \quad x \in \Omega  \\
\hspace{-4.2cm} \left. u\right|_{\delta\Omega} = 0
\end{matrix} \right.,
\label{eq:lap}
\end{equation}
where $\Omega = [0,1]^3,$ {the diffusion tensor $k(x) = \mathrm{diag}(x_1^2+0.5, x_2^2+0.5,x_3^2+0.5)$, right hand side $f(x) = 1$.
CE algorithm is implemented { in Fortran using BLAS from Intel MKL.}
For MINRES we use Python library SciPy.}
Computations were performed on a server with 32 Intel\textsuperscript{\textregistered} Xeon\textsuperscript{\textregistered} E5-2640 v2 (20M Cache, 2.00 GHz) processors and with 256GB of RAM. Tests were run in a single-processor mode.

\subsection{{CE($\varepsilon$) factorization as a direct solver}}
The factorization obtained in CE($r$) procedure is not accurate enough to use it as a direct solver. Consider CE($\varepsilon$) factorization as an approximate direct solver.
In Table \ref{tab:err} we show the relative accuracy $\eta$ that can be achieved by a direct solution of the system with { a matrix approximated by }CE($\varepsilon$) factorization.
\begin{table}[H]
\centering
\begin{tabular}{ l | c  c  c  c }
Revealing parameter,~$\varepsilon$ & $10^{-2}$ & $10^{-4}$ & $10^{-6}$ & $10^{-8}$\\
\hline
Relative accuracy,~$\eta$& $4.0 \times 10^{-1}$ & $9.1 \times 10^{-3}$& $1.2 \times 10^{-5}$& $9.9 \times 10^{-7}$\\
Required memory,~MB & 553 & 1348 & 2494& 2671 \\
{Time,~sec} &  3.40107 &8.28388& 17.40455 & 29.65910\\
\end{tabular}
\caption{Factorization accuracy and required memory, $N = 65536$. Accuracy $\eta$ here is computed as: $\eta = \frac{||x_{CE} - x_*||}{||x_*||}$, where $x_{CE}$  is the obtained the solution and~$x_*$ is exact solution.}
\label{tab:err}
\end{table}
The main problem with the direct CE($\varepsilon$) solver is that it requires too much memory. On the other hand, CE($\varepsilon$) with large enough~$\varepsilon$ works fast and requires small amount of memory. This factorization as well as the CE($r$) factorization can be utilized as a very good preconditioner for iterative solvers, e.g. MINRES.
\subsection{CE($\varepsilon$), CE($r$) convergence}
In Figure~\ref{fig:conv} the convergence of MINRES with CE($\varepsilon$), CE($r$), ILUt preconditioners and MINRES without preconditioner are shown.

\begin{figure}[H]
\centering
\begin{subfigure}{.5\textwidth}
  \centering
  \scalebox{0.25}{
    \includegraphics{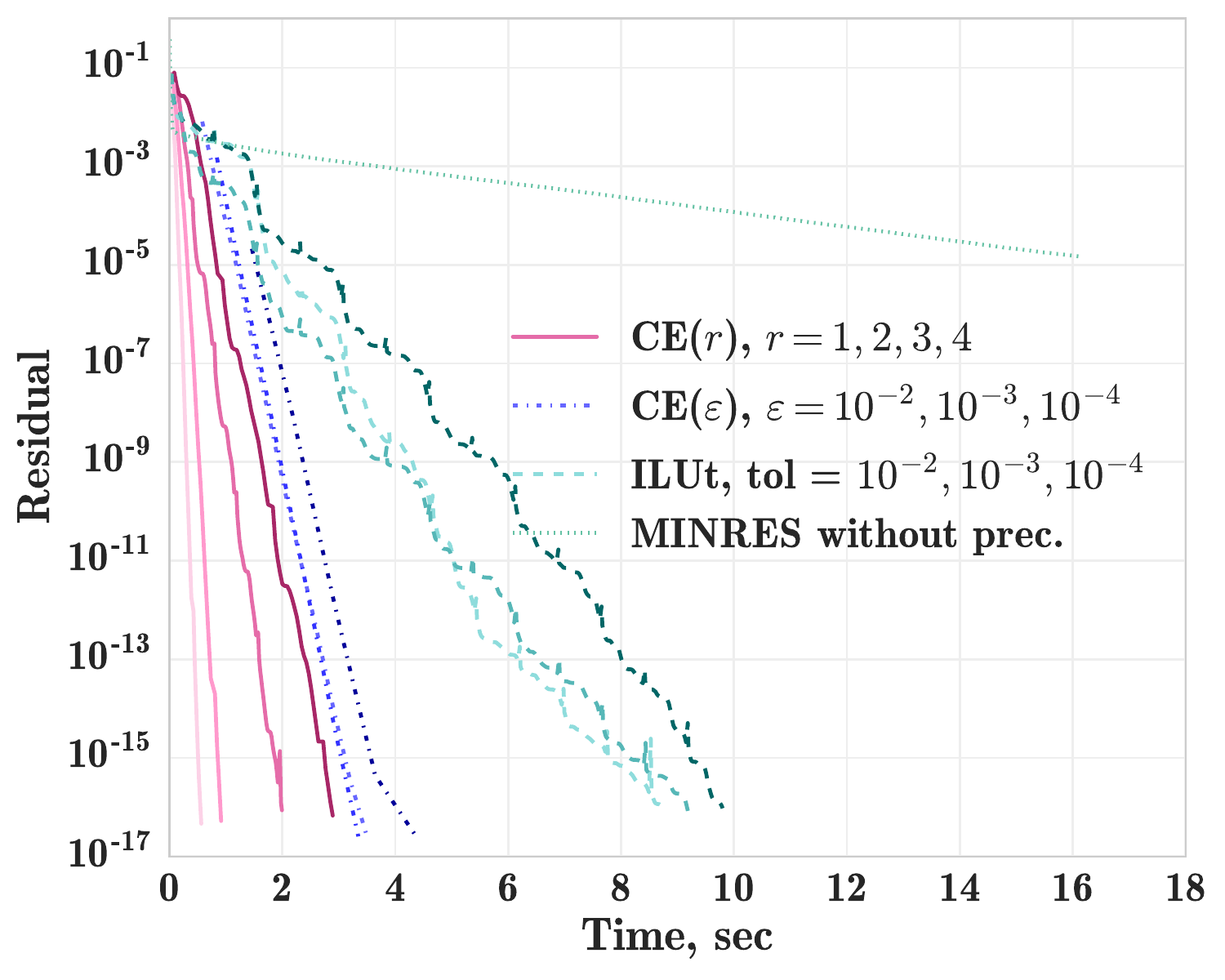}}
  \caption{Residual/time}
  \label{fig:hist_eps_r1}
\end{subfigure}%
\begin{subfigure}{.5\textwidth}
  \centering
  \scalebox{0.25}{
\includegraphics{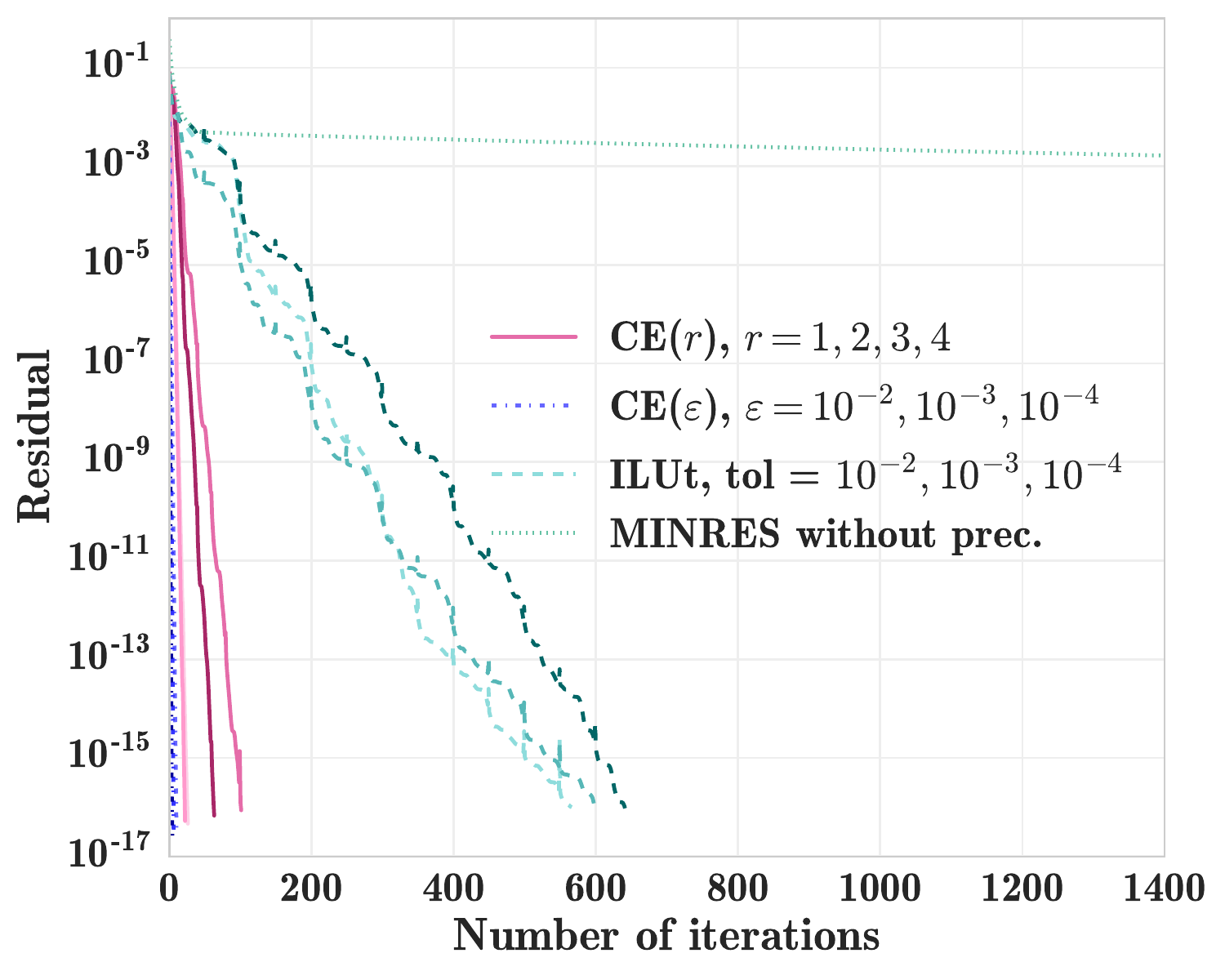}}
  \caption{Residual/iterations}
  \label{fig:hist_eps_r2}
\end{subfigure}
\caption{Convergence comparison of MINRES with different preconditioners, $N~=~32768$. {Note that the line style and the line color on this figure mark the type of the preconditioner, different parameters have different color tones, from dark to light in the list of parameters.}}
        \label{fig:conv}
\end{figure}$\\$
{
\begin{remark}
Note that the number of MINRES iterations without the preconditioner as well as the number of MINRES iterations with the ILUt preconditioner grows as the mesh size increases, because the matrix becomes increasingly ill-conditioned. CE($\varepsilon$) and CE($r$) are much better preconditioners. For CE($\varepsilon$) preconditioner with fixed accuracy number of iterations does not grow; for CE($r$) preconditioner growth of the number of iterations is very mild (Table~\ref{tab:num_it}).
\end{remark}

\begin{table}[H]
\centering
\begin{tabular}{ l | c  c  c  c  c }
Matrix size,~$N$& 8192 & 16384 & 32768 & 65536 & 131072\\
\hline
Without preconditioner & 147 & 2636 & $>1000$ & $>1000$ & $>1000$ \\
ILUt, $t=10^{-3}$    & 40 & 92 & 339 & $>1000$ & $>1000$ \\
CE($r$), $r = 4$  & 23 & 25 & 29 & 30 & 36 \\
CE($\varepsilon$), $\varepsilon=10^{-3}$ & 4 & 5 & 6 & 5 & 6\\
\end{tabular}
\caption{The number of iterations required for MINRES to the converge to accuracy $\epsilon = 10^{-10}$ for different preconditioners.}
\label{tab:num_it}
\end{table}
}

{ We have} the standard trade off: the smaller the $\varepsilon$ is, the better the convergence is, but the storage grows. This fact is illustrated in Figure~\ref{fig:hist_eps_r}.

\begin{figure}[H]
\centering
\begin{subfigure}{.5\textwidth}
  \centering
  \scalebox{0.25}{
    \includegraphics{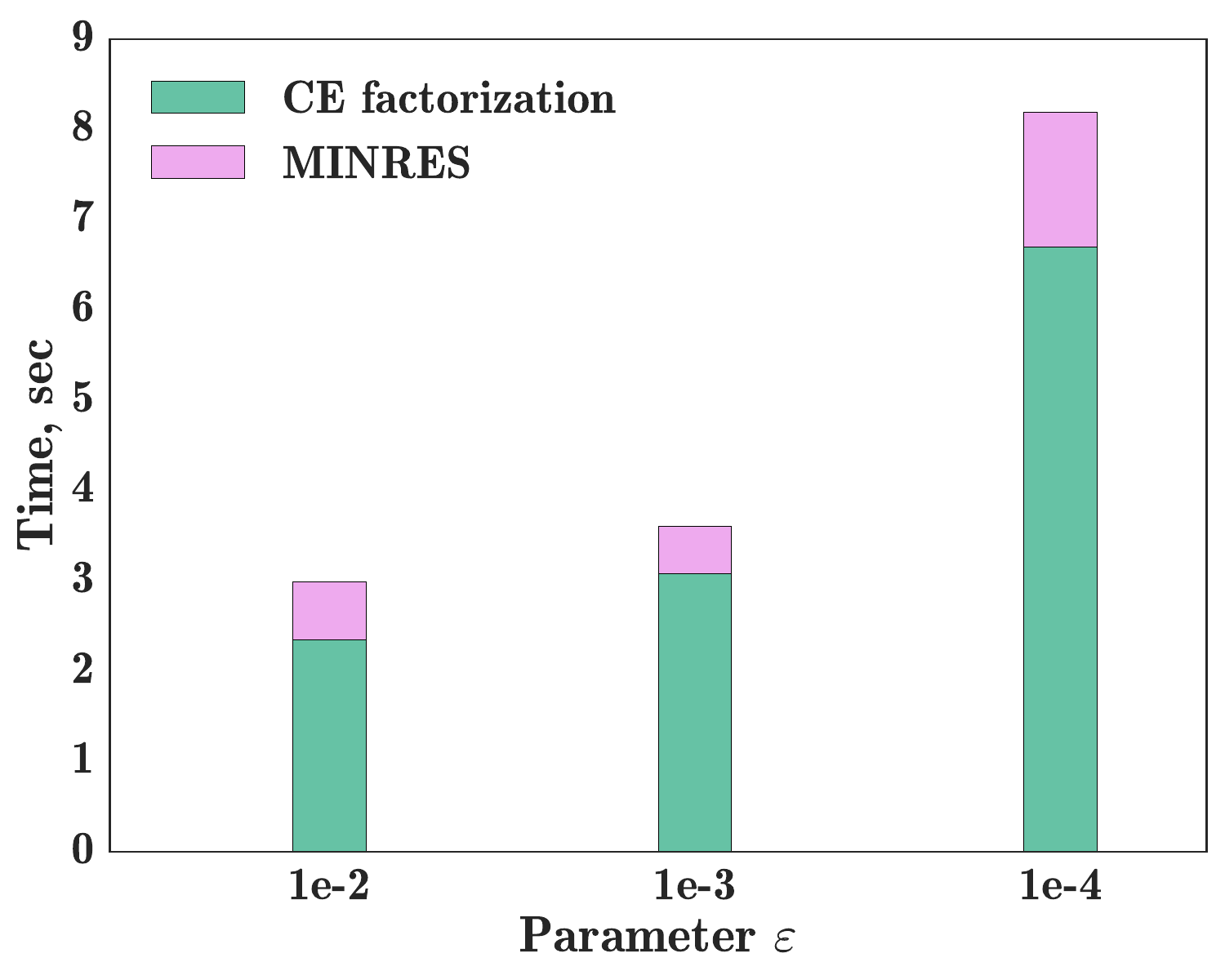}}
  \caption{Adaptive preconditioner CE($\varepsilon$)}
  \label{fig:hist_eps_r1}
\end{subfigure}%
\begin{subfigure}{.5\textwidth}
  \centering
  \scalebox{0.25}{\includegraphics{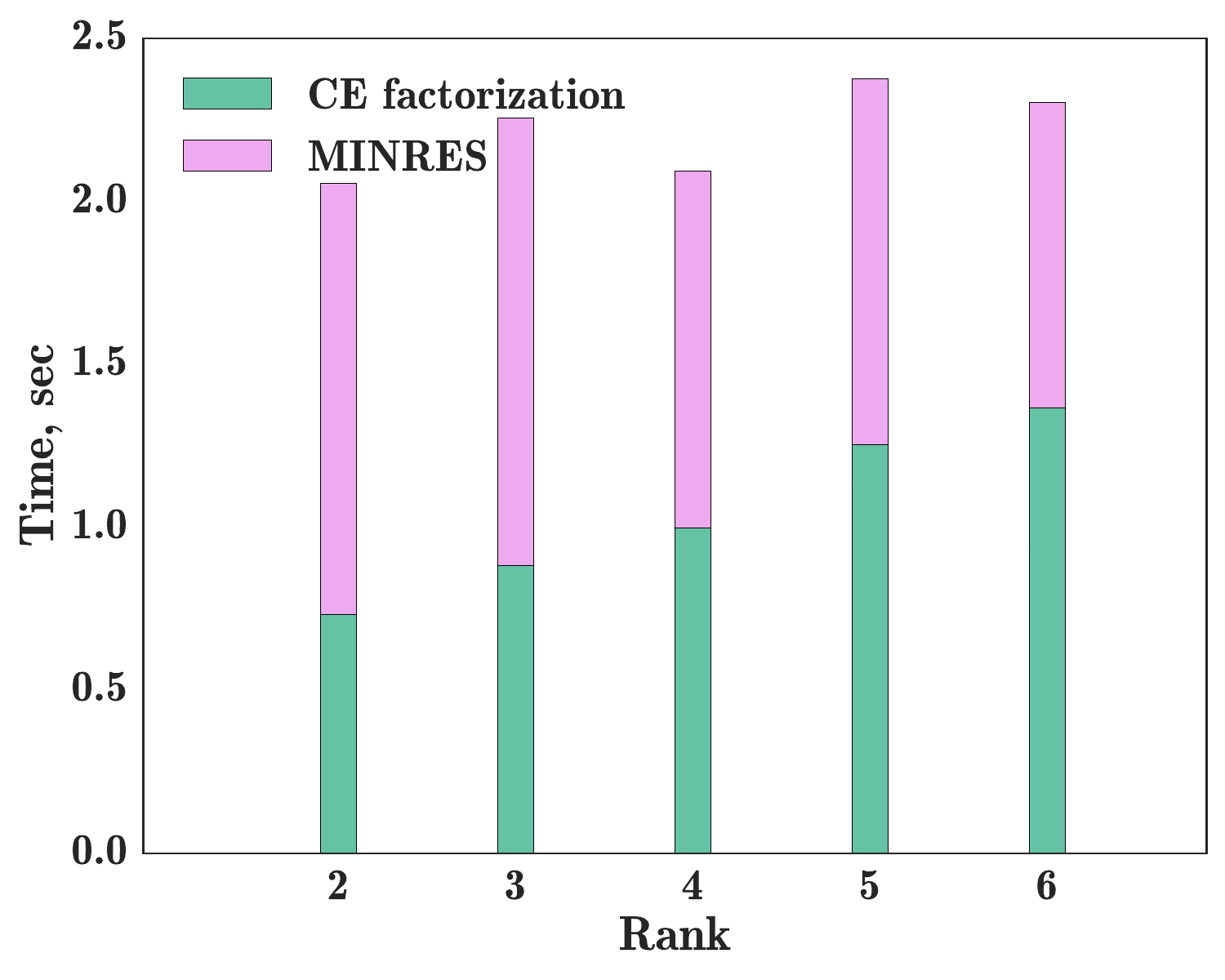}}
  \caption{Fixed-rank preconditioner CE($r$)}
  \label{fig:hist_eps_r2}
\end{subfigure}
\caption{Factorization and MINRES time, $N = 32768$}
\label{fig:hist_eps_r}
\end{figure}$\\$

Note that in Figure~\ref{fig:hist_eps_r2} the total time for different ranks is the same, but { for} $r=2$ less memory is required.
Also note that CE($\varepsilon$) and CE($r$) have comparable timings so let us compare their memory requirements for different $N$.

\begin{figure}[H]
\centering
\begin{subfigure}{.5\textwidth}
  \centering
  \scalebox{0.25}{\includegraphics{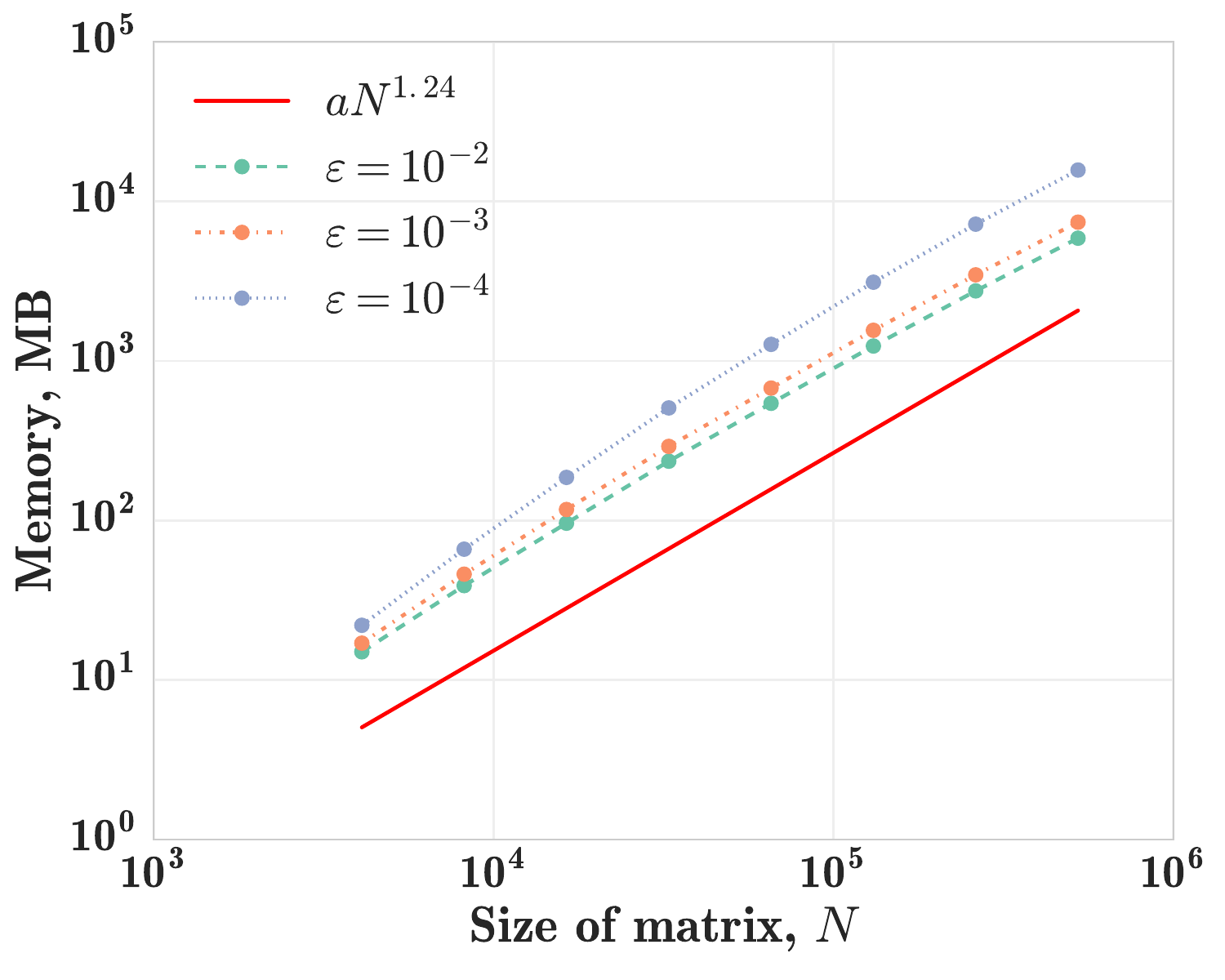}}
  \caption{Adaptive preconditioner CE($\varepsilon$)}
  \label{fig:sub1}
\end{subfigure}%
\begin{subfigure}{.5\textwidth}
  \centering
  \scalebox{0.25}{\includegraphics{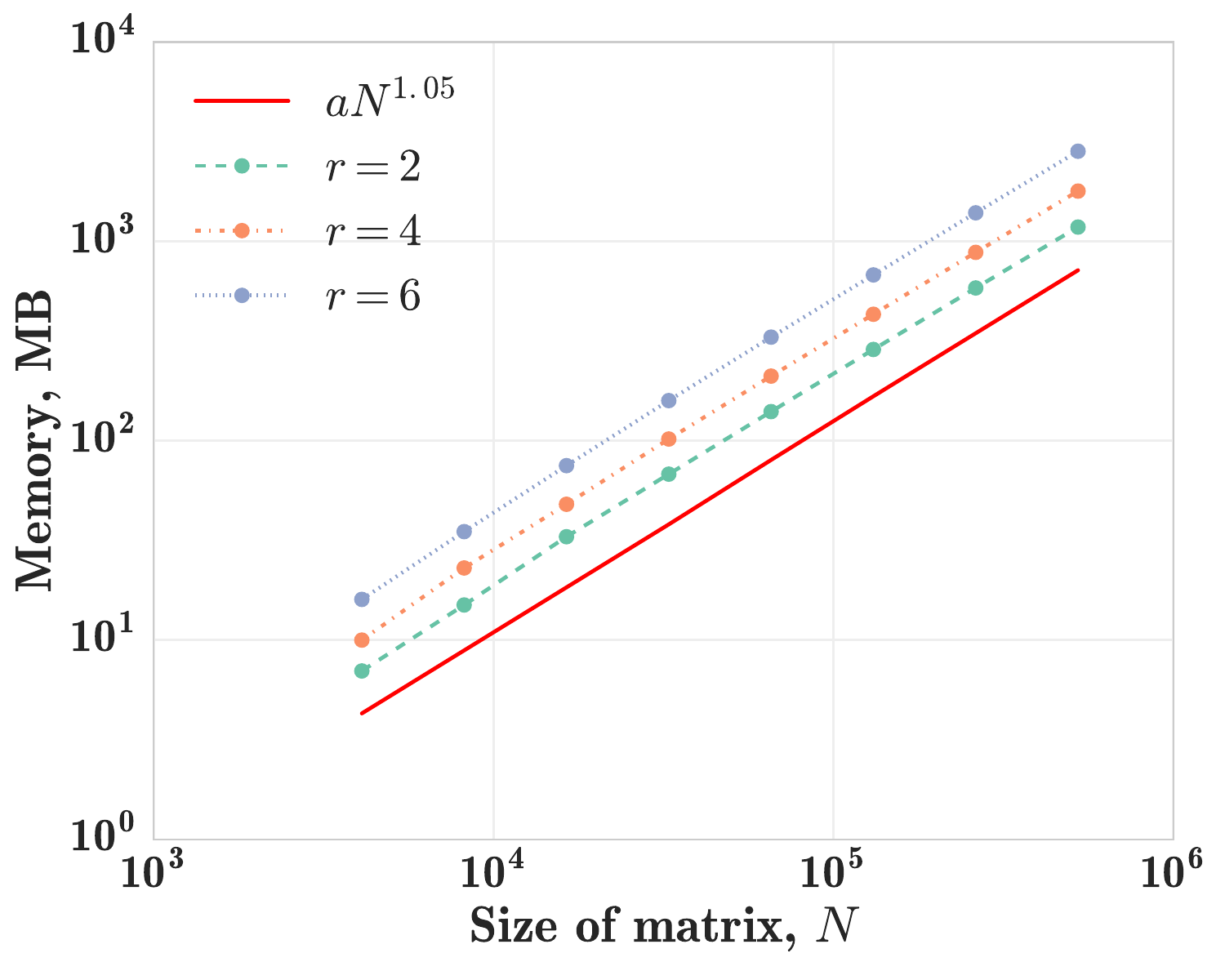}}
  \caption{Fixed-rank preconditioner CE($r$)}
  \label{fig:sub2}
\end{subfigure}
\caption{Memory requirements for approximate factorization}
\label{fig:test}
\end{figure}$\\$

The memory requirements for the fixed-rank preconditioner are predictably lower than for the adaptive ones. In Figure \ref{fig:tot_sol} the total time required for CE($\varepsilon$) and CE($r$) preconditioners and MINRES are shown for different $N$.
\begin{figure}[H]
\centering
\begin{subfigure}{.5\textwidth}
  \centering
  \scalebox{0.25}{\includegraphics{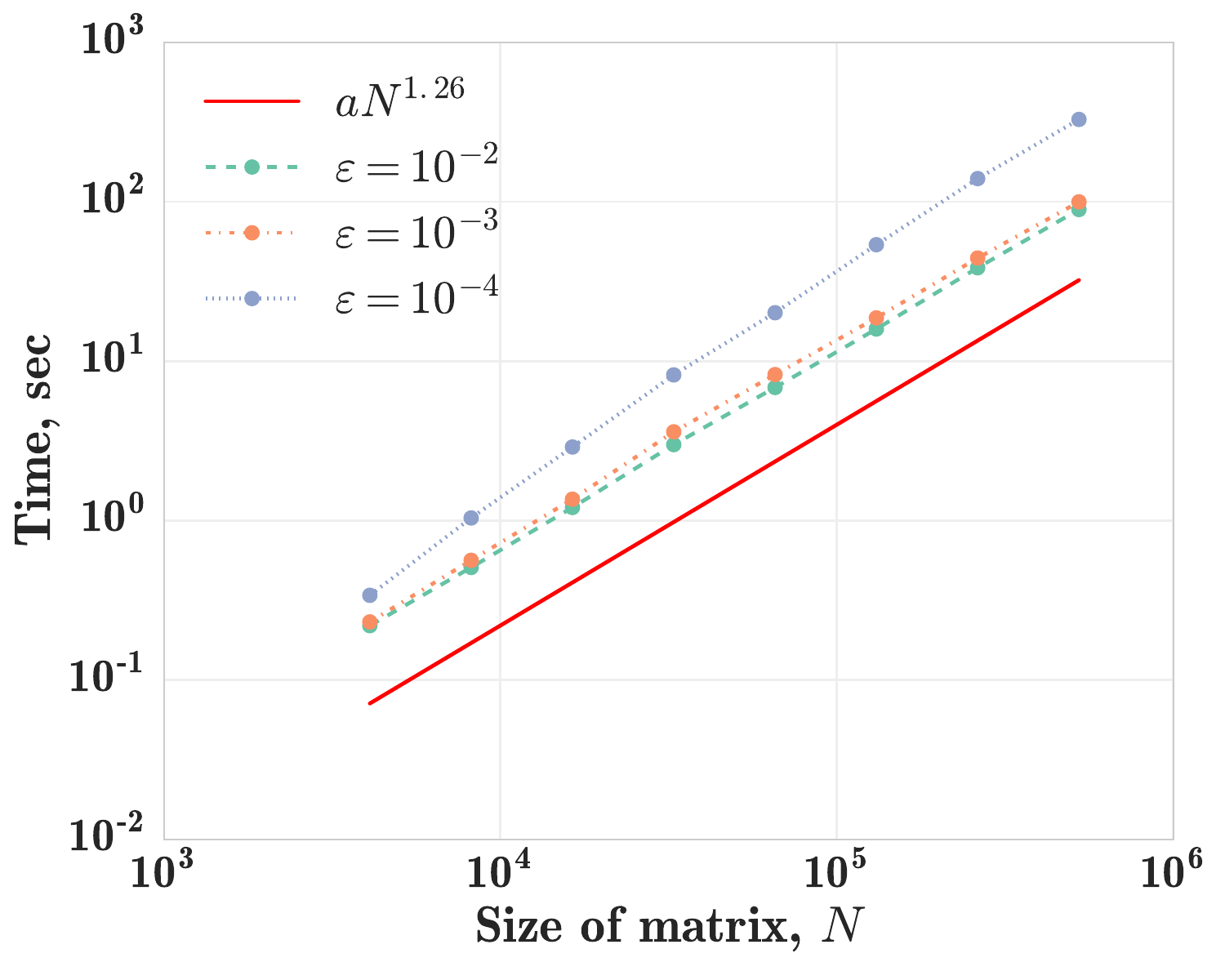}}
  \caption{Adaptive preconditioner CE($\varepsilon$)}
  \label{fig:sub1}
\end{subfigure}%
\begin{subfigure}{.5\textwidth}
  \centering
  \scalebox{0.25}{\includegraphics{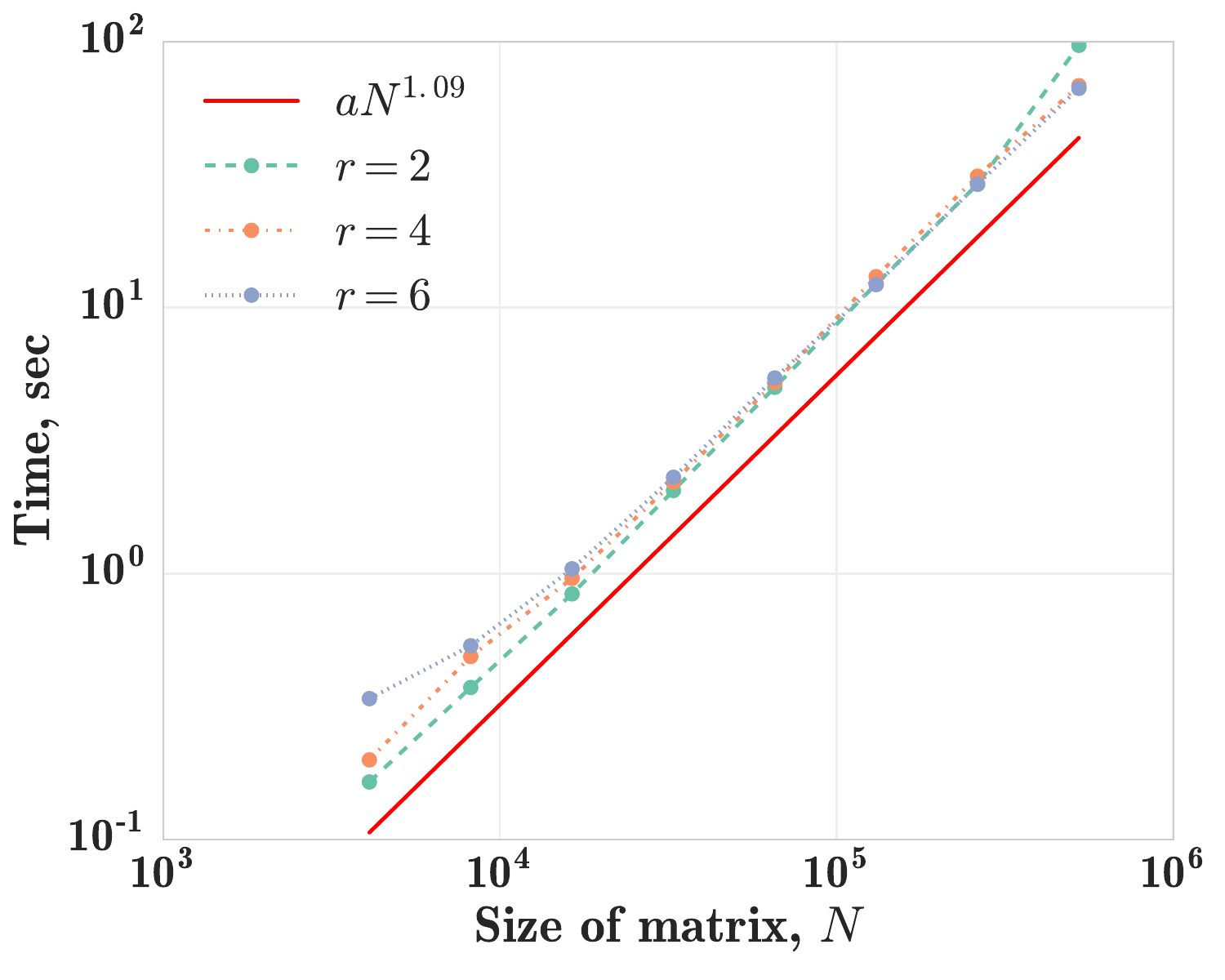}}
  \caption{Fixed-rank preconditioner CE($r$)}
  \label{fig:sub2}
\end{subfigure}
\caption{Total solution time for MINRES with CE($\varepsilon$) and CE($r$) preconditioners}
\label{fig:tot_sol}
\end{figure}$\\$

We have found experimentally that for the CE($r$) preconditioner $r = 4$ is typically a good choice, note that in this case $r = \frac{B}{2}$. For the CE($\varepsilon$) preconditioner a good choice is $\varepsilon = 10^{-2}$ { (for this particular example).}

\subsubsection{Final comparison}
Let us now compare CE($\varepsilon$) solver, CE($r$) solver, CHOLMOD and MINRES with ILUt preconditioner. {Since MINRES with ILUt does not converge in 1000 iterations for $N>32768$, we show its performance only for the points where it converged.}

\begin{figure}[H]
\centering
\begin{subfigure}{.5\textwidth}
  \centering
  \scalebox{0.25}{\includegraphics{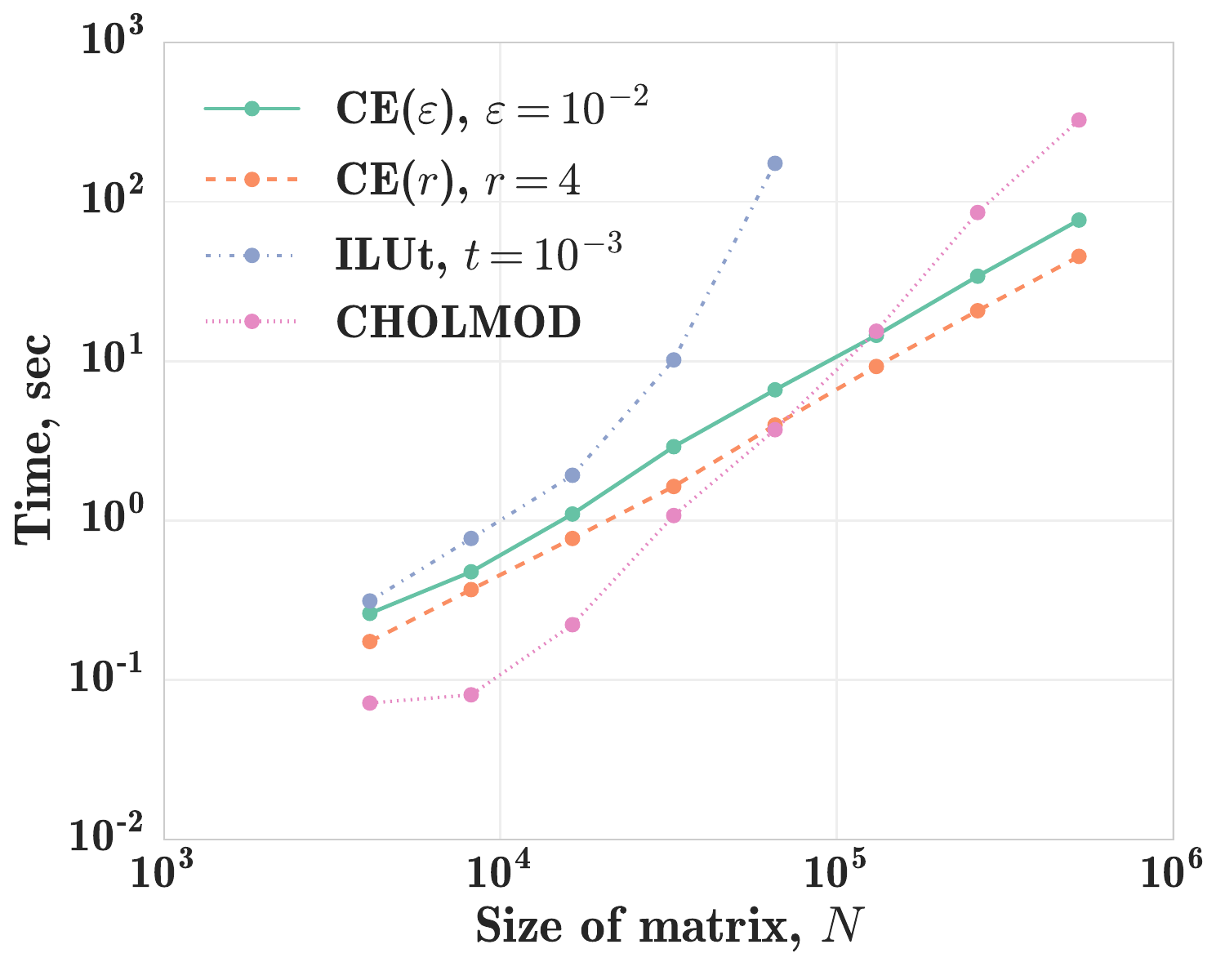}}
  \caption{Time comparison.}
  \label{fig:fin_t}
\end{subfigure}%
\begin{subfigure}{.5\textwidth}
  \centering
  \scalebox{0.25}{\includegraphics{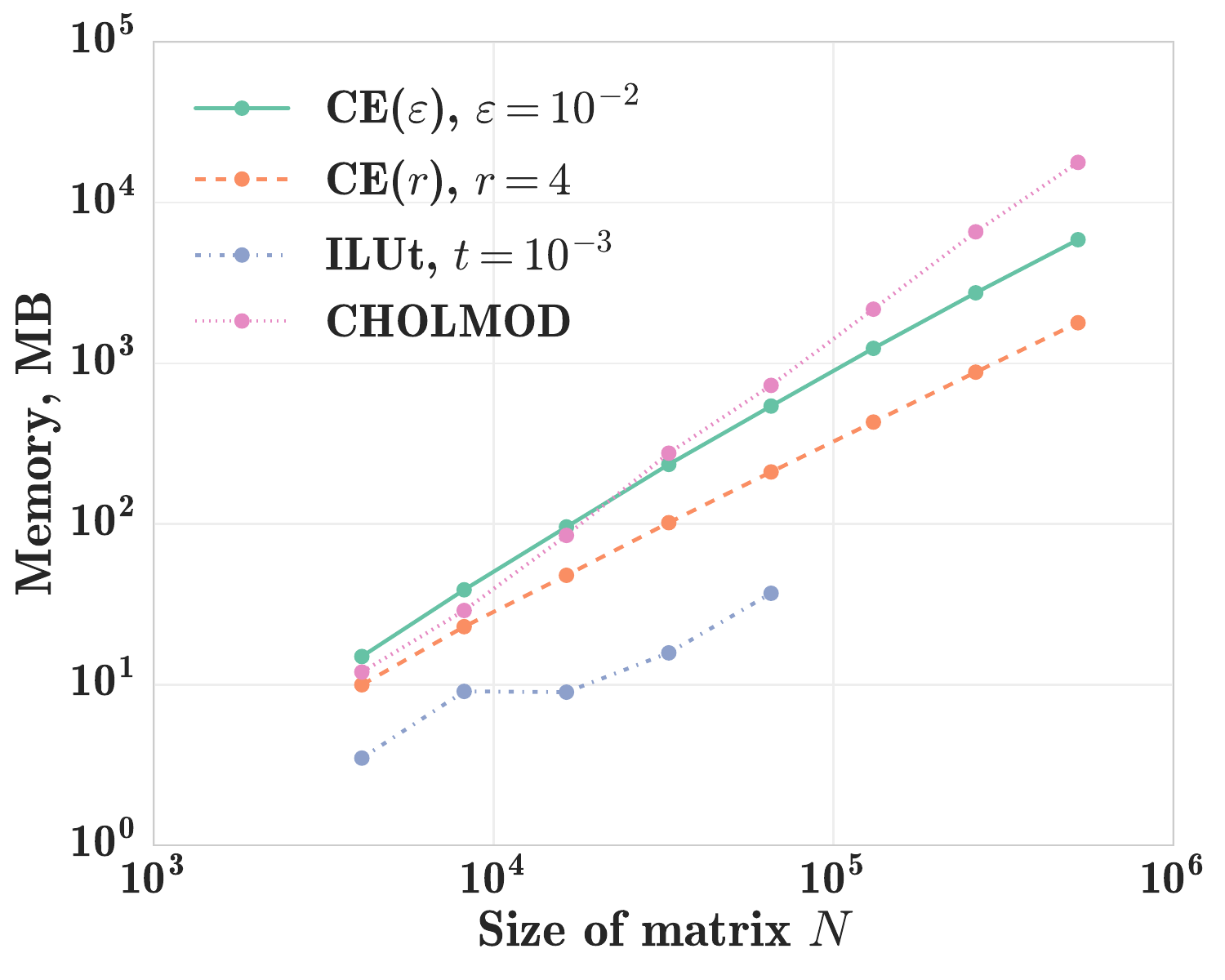}}
  \caption{Memory comparison}
  \label{fig:fin_m}
\end{subfigure}
\caption{Final solvers comparison (note logarithmic scales).}
\label{fig:fin}
\end{figure}$\\$

Note that CE($r$) solver starts to use less memory than CHOLMOD at $N \approx 5000$ and starts to work faster at  $N \approx 10000$. For the largest $N$ our hardware permits we are able to solve the system $\sim10$ times faster than CHOLMOD, and require $\sim10$ times less memory. Note that the accuracy of CHOLMOD solution is $10^{-11}$.

\section{Related work}

There are two big directions in fast solvers for sparse linear systems: sparse (approximate) factorizations and low-rank based solvers. CE-algorithm has the flavor of both.

The main difference between CE-algorithm and the classical direct sparse LU-like solvers {(e.g. block Gaussian elimination \cite{bank-blockLU-1976},
multifrontal method \cite{Aminfar-multifrontal-2015} and
others \cite{davis-UMFPACK-2004,ng-supernodal-1993,Davis-cholmod-2008,chadwick-rank_str_chol-2015}) is that in our algorithm fill-in growth is controlled while maintaining the accuracy, thanks to the additional compression procedure.
This leads to big advantage in memory usage and probably makes CE-algorithm asymptotically faster, but this point requires additional analysis. Note that references \cite{Aminfar-multifrontal-2015} and \cite{chadwick-rank_str_chol-2015} use compressed representations (i.e.
they are not just using the graph structure of the matrix, as is the case
for the other solvers).}

CE-algorithm is close to \emph{hierarchical solvers}.
These solvers utilize the \emph{hierarchical structure} of matrices $A$ and/or inverse matrix $A^{-1}$.
For example  $\mathcal{H}$ direct solvers \cite{Hakb-h-lib,martinsson-h_solver-2009,bebendorf-existence-2003, Beb-hlu-2005, grasedyck-hlu-2008} build the $LU$ factorization and/or the approximation of the inverse $A^{-1}$ in the $\mathcal{H}$ and $\mathcal{H}^2$ formats. { Another example of hierarchical inversion methods is divide and conquer sparse preconditioning} \cite{li-di_and_con-2012}. This type of solvers has provable linear complexity for matrices coming as discretization of PDEs, but the constant hidden in $\mathcal{O}(N)$ can be quite large, making them non-competitive.
{Hierarchical inversion methods ($\mathcal{H}$-LU method, divide and conquer sparse preconditioning, etc.) and CE solver are based on similar high-level ideas, but the algorithms are different:hierarchical inversion solvers utilize recursion, while CE algorithm is going from smallest to largest blocks.}

Recently, so-called \emph{HSS-solvers} have attracted a lot of attention. To name few references: \cite{chandrasekaran-HSSsolver-2006,ChMing-dir_hss-2006, Ho-dir_hss-2012,rouet-hss-2015,martinsson-hss_integr-2005}. As a subclass of such solvers, $HODLR$ direct solvers \cite{ambikasaran-hodlersolver-2013,aminfar-lodlersolver-2016,kong-lodlersolver-2011} have been introduced. These solvers compute approximate sparse LU-like factorization of HSS-matrices. The simplicity of the structure in comparison to the general $\mathcal{H}^2$ structure allows for a very efficient implementation, and despite the fact that these matrices are essentially 1D-$\mathcal{H}^2$ matrices and optimal linear complexity is not possible for the matrices that are discretization of 2D/3D PDEs the running times can be quite impressive.

{This work is closely related to the work by Ying and Ho with skeletonization: they use similar idea but compress all off-diagonal blocks, while we compress only the far away ones \cite{ho-skelet_i-2015,ho-skelet_d-2015, li-int-2016}. }

The $\mathcal{H}$-matrices have close connection to the fast multipole method (FMM), and an \emph{inverse FMM solver} \cite{Ambikasaran-dir-h2-2014,coulier_ifmm_2015} aimed
at the inversion of integral transforms given by the FMM-structure. This method can be also adapted to the sparse matrix case as in paper \cite{pouransari_ss_2015}. The CE-algorithm has a simpler logic and structure, similar to the standard LU-factorization techniques accompanied by the compression procedure.

\section{Conclusions and future work}
We have proposed a new approximate factorization for sparse SPD matrices that can be computed with linear cost and storage, but provides better approximation to the matrix than standard incomplete LU factorization techniques. The constant is quite small compared to other low-rank based approaches. In our model experiments we were able to outperform CHOLMOD software in terms of memory cost and computational time. There are several directions for future research. At the moment, the permutation is assumed to be known by the start of the algorithm. In many cases it can be retrieved from the geometric information. However, it would be interesting to develop the version that computes the next row to be eliminated in the adaptive way.
Another important topic is to extend the solver for the non-symmetric case and also for dense $\mathcal{H}^2$ matrices. It is quite straightforward, but a lot of questions are still open. Finally, the efficient implementation of CE requires many small optimization of the original prototype code: in the current version most of the time is spent in the SVD procedure in the compression step, and using a cheaper rank-revealing factorization may significantly reduce the computational time. In the end, a much more broad comparison with other available solvers is needed to see the limitations of our approach in practice, as well as the theoretical justification of the low-rank property of the far blocks that appear in the process.

\section*{Acknowledgments}

The authors would like to thank Maxim Rakhuba, Igor Ostanin, Alexander Novikov, Alexander Katrutsa and Marina Munkhoeva for {\color{royalblue} their comments on an earlier version of this paper. We would also like to show our gratitude to the anonymous reviewer for the idea to illustrate the CE algorithm with expanded residual graphs, Section~\ref{sec:erg} is based on this idea.}

\bibliography{lib}
\end{document}